\documentclass[a4paper,10pt,fleqn]{elsarticle}

\usepackage[latin1]{inputenc}    
\usepackage[T1]{fontenc}      
\usepackage{geometry}         
\usepackage[frenchb, english]{babel}  
\usepackage{amsfonts,amssymb,amsmath,amsthm}
\usepackage{geometry}
\usepackage{mathrsfs}

\usepackage{dsfont}
\usepackage{appendix}
\usepackage{graphicx}

\newtheorem{théo}{Theorem}[section]
\newtheorem{prop}[théo]{Proposition}
\newtheorem{lemm}[théo]{Lemma}
\newtheorem{corr}[théo]{Corollary}

\theoremstyle{definition}
\newtheorem{rema}[théo]{Remark}
\newtheorem{hypo}{Hypothesis}[section]

\def\transpose{{}^t\!}
\DeclareMathOperator{\tr}{Tr}
\newcommand{\E}{\mathbb{E}}
\newcommand{\M}{\mathscr{M}^2(\mathbb{R}_+,\mathbb{R}^{1 \times d})}
\newcommand{\der}{\mathrm{d}}
\newcommand{\R}{\mathbb{R}}
\newcommand{\Rd}{\mathbb{R}^d}

\newcommand{\G}{\overline{G}}
\newcommand{\Pb}{\mathbb{P}}
\newcommand{\N}{\mathbb{N}}
\newcommand{\Q}{\mathbb{Q}}
\newcommand{\Pt}{\mathscr{P}_t}

\newcommand{\LL}{\mathscr{L}}
\newcommand{\OO}{\mathscr{O}}
\newcommand{\SSS}{\mathscr{S}^2}
\newcommand{\F}{\mathscr{F}}

\renewcommand{\t}[1]{\text{#1}}

\numberwithin{equation}{section}

\begin{document}

\begin{frontmatter}

\title{Ergodic BSDEs and related PDEs with Neumann boundary conditions under weak dissipative assumptions}

 \author[pym]{P.Y. Madec\corref{cor1}}
 \ead{pierre-yves.madec@univ-rennes1.fr}

\cortext[cor1]{Tel.:+33 02 23 23 53 05}
 
\address[pym]{IRMAR, Université de Rennes 1, Campus de Beaulieu, 35042 Rennes Cedex, France.}

\begin{abstract}
We study a class of ergodic BSDEs related to PDEs with Neumann boundary conditions. The randomness of the driver is given by a forward process under weakly dissipative assumptions with an invertible and bounded diffusion matrix. Furthermore, this forward process is reflected in a convex subset of $\R^d$ not necessarily bounded. We study the link of such EBSDEs with PDEs and we apply our results to an ergodic optimal control problem.
\end{abstract}
 
\end{frontmatter}

\section{Introduction}
In this paper we study the following ergodic backward stochastic differential equation (EBSDE in what follows) in finite dimension and in infinite horizon: $\forall t,T \in \R_+, ~~0 \leq t \leq T <+\infty$: 
\begin{align}\label{EBSDE intro Neumann}
Y_t^x = Y_T^x + \int_t^T [\psi(X_s^x,Z_s^x) - \lambda] \der s + \int_t^T [g(X_s^x)-\mu]\der K_s^x - \int_t^T Z_s^x \der W_s,
\end{align}
where the given data satisfy:
\begin{itemize}
\item $W$ is an $\R^d$-valued standard Brownian motion;
\item $G = \{ \phi > 0 \}$ is an open convex subset of $\R^d$ with smooth boundary;
\item $x \in G$;
\item $X^x$ is a $\G$-valued process starting from $x$, and $K^x$ is a non decreasing real valued process starting from $0$ such that the pair $(X^x,K^x)$ is a solution of the following reflected stochastic differential equation (SDE in what follows):\\ \vspace{-0,5 cm}
\begin{align*}
&X_t^x = x + \int_0^t f(X_s^x) \der s + \int_0^t \sigma (X_s^x) \der W_s + \int_0^t \nabla \phi(X_s^x) \der K_s^x, ~~t \geq 0,\\
&K_t^x = \int_0^t \mathds{1}_{\{ X_s^x \in \partial G \}} \der K_s^x, ~~ K^x_{\cdot} \t{~ is non-decreasing},
\end{align*}
\item $\psi: \R^d \times \R^{1 \times d} \rightarrow \R$ is  measurable and $g : \R^d \rightarrow \R$ is measurable; 
\item $\lambda$ and $\mu$ belong both to $\R$.  If $\lambda$ is given then $\mu$ is unknown and if $\mu$ is given then $\lambda$ is unknown.
\end{itemize}
Therefore, the unknown is either the triplet $(Y^x,Z^x,\lambda)$ if $\mu$ is given or the triplet $(Y^x,Z^x,\mu)$ if $\lambda$ is given, where:
\begin{itemize}
\item $Y^x$ is a real-valued progressively measurable process;
\item $Z^x$ is an $\R^{1 \times d}$-valued progressively measurable process.
\end{itemize}

We recall that a function $h : \R^d \rightarrow \R^d$ is said to be strictly dissipative if there exists a constant $\eta > 0$ such that, $\forall x,y \in \R^d$,
\begin{align*}
(h(x)-h(y) , x-y) \leq -\eta|x-y|^2.
\end{align*}

Richou in the paper \cite{Richou_Ergodic_BSDE_PDe_Neumann} studied the case when $\G$ is bounded and with the assumptions that $f$ and $\sigma$ are Lipschitz and:
\begin{align*}
\sup_{x,y \in \G, x \neq y} \left\{ \frac{\transpose(x-y)(f(x)-f(y))}{|x-y|^2} + \frac{\tr[(\sigma(x)-\sigma(y)) \transpose(\sigma(x)-\sigma(y))]}{2|x-y|^2} \right\} < -K_{\psi, z}K_{\sigma}
\end{align*}
where $K_{\psi, z}$ is the Lipschitz constant of $\psi$ in $z$ and $K_\sigma$ is the Lipschitz constant of $\sigma$. Note that this assumption implies that $f$ is strictly dissipative. However this hypothesis on $f$ is not very natural because it supposes a dependence between parameters of the problem. Thanks to this condition it is possible to establish one of the key results: the strong estimate on the exponential decay in time of two solutions of the forward equation starting from different points. Indeed, it is used to construct, by a diagonal procedure, a solution to the EBSDE. Note that, in this work, $\G$ is assumed to be bounded. 

In the paper  \cite{Debu_Hu_Tess_weak_dissipative}, Debussche, Hu and Tessitore  were concerned with the study of EBSDE in a weakly dissipative environment. This means that the driver of the forward process is assumed to be the sum of a strictly dissipative term and a perturbation term which is Lipschitz and bounded. In their infinite dimensional framework, they supposed that the dissipative term is linear. In addition, $\sigma$ is constant, and the forward process is not reflected. Finally the coefficients of the forward process are assumed to be Gâteaux differentiable to obtain an estimate which is needed to prove the existence of a solution in this framework. In this context, the weaker assumption on $f$ makes the strong estimate on the exponential decay in time of two solutions of the forward equation impossible. However it is possible to substitute this result by a weaker result, called "basic coupling estimate" which involves the Kolmogorov semigroups of the forward process  $X^x$ and which is enough to prove the existence of a solution to the EBSDE. 

In this paper we extend the framework of \cite{Richou_Ergodic_BSDE_PDe_Neumann} to the case of an unbounded domain $\G$ for a driver weakly dissipative. Namely, we assume that $f=d+b$ where $d$ is locally Lipschitz and  dissipative with polynomial growth and $b$ is Lipschitz and bounded. The price to pay is that $\sigma$ is assumed to be Lipschitz, invertible and such that $\sigma$ and $\sigma^{-1}$ are bounded. We do not need more regularity than continuous coefficients for this study, because we treat this problem by a regularization procedure. As the basic coupling estimate of \cite{Debu_Hu_Tess_weak_dissipative} holds for a non reflected process, we start by studying the following forward process, $\forall t \geq 0$,
\begin{align*}
&V_t^x = x + \int_0^t f(V_s^x) \der s + \int_0^t \sigma (V_s^x) \der W_s,
\end{align*}
with $f $ and $\sigma$ defined as before. We show that the coupling estimate still holds in our framework with constants which depend on $d$ only through its dissipativity coefficient. Once this is established, we apply this result to establish existence and uniqueness (of $\lambda$) of solutions to the following EBSDE:
\begin{equation}\label{EBSDE général avec f intro}
Y_t^x = Y_T^x + \int_t^T[ \psi(V_s^x,Z_s^x)- \lambda] \der s - \int_t^T Z_s^x \der W_s, ~ \forall 0 \leq t \leq T < +\infty.
\end{equation}
Then we want to obtain the same result when the process $V_\cdot^x$ is replaced by a reflected process $X_\cdot^x$ in $\G$, namely:
\begin{equation}\label{EBSDE avec process reflechi non neumann intro}
Y_t^x = Y_T^x + \int_t^T[ \psi(X_s^x,Z_s^x)- \lambda] \der s - \int_t^T Z_s^x \der W_s, ~ \forall 0 \leq t \leq T < +\infty.
\end{equation}
 For this purpose, we use a penalization method to construct a sequence of processes $X^{x,n}$ defined on the whole $\R^d$ and which converges to the reflected process $X^x$.  More precisely, we denote by $(Y^{x,\alpha,n,\varepsilon},Z^{x,\alpha,n,\varepsilon})$ the solution of the following BSDE with regularized coefficients $\psi^{\varepsilon}$, $d^{\varepsilon}$, $F_n^{\varepsilon}$ and $b^{\varepsilon}$ by convolution with a sequence approximating the identity, $\forall t,T \in \R_+, ~~0 \leq t \leq T <+\infty$:
\begin{align}
Y_t^{x,\alpha,n,\varepsilon} = Y_T^{x,\alpha,n,\varepsilon} + \int_t^T [\psi^{\varepsilon}(X_s^{x,n,\varepsilon},Z_s^{x,\alpha,n,\varepsilon}) - \alpha Y_s^{x,\alpha,n,\varepsilon}] \der s - \int_t^T Z_s^{x,\alpha,n,\varepsilon} \der W_s,
\end{align}
where $X^{x,n,\varepsilon}$ is the strong solution of the SDE:
\begin{align*}
X_t^{x,n,\varepsilon} = x + \int_0^t (d^{\varepsilon} + F_n^{\varepsilon} +b^{\varepsilon})(X_s^{x,n,\varepsilon})\der s + \int_0^t \sigma^{\varepsilon}(X_s^{x,n,\varepsilon}) \der W_s.
\end{align*}
Note that as $F_n$ is dissipative with a dissipative constant equal to $0$, $d + F_n$ remains dissipative with a dissipative coefficient equal to $\eta$. Then, making  $\varepsilon \rightarrow 0$, $n \rightarrow + \infty$ and $\alpha \rightarrow 0$, it is possible to show that, roughly speaking, $(Y_t^{x,\alpha,n,\varepsilon}-Y_0^{x,\alpha,n,\varepsilon},Z_t^{x,\alpha,n,\varepsilon},Y_0^{x,\alpha,n,\varepsilon}) \rightarrow (Y_t^x,Z_t^x,\lambda)$ which is solution of EBSDE  (\ref{EBSDE avec process reflechi non neumann intro}). Once a solution $(Y,Z,\lambda)$ is found for the EBSDE (\ref{EBSDE avec process reflechi non neumann intro}) we study existence and uniqueness of solutions of the type $(Y,Z,\lambda)$ and $(Y,Z,\mu)$ of the EBSDE (\ref{EBSDE intro Neumann}). Here we only manage to find solutions which are not Markovian and which are not bounded in expectation. Then we show that the function defined by $v(x) := Y_0^x$, where $Y$ is a solution of EBSDE (\ref{EBSDE avec process reflechi non neumann intro}) is a viscosity solution of the following partial differential equation (PDE in what follows) :
\begin{align}
\left\{ \begin{array}{ll}
\LL v(x) + \psi(x, \nabla v(x)\sigma(x)) = \lambda, &x \in G, \\
\frac{\partial v}{\partial n}(x)  = 0, & x \in \partial G,
\end{array}\right.
\end{align}
where:
\begin{align*}
\LL u(x) = \frac{1}{2}\tr (\sigma(x) \transpose \sigma(x) \nabla^2u(x)) + \transpose f(x)\nabla u(x).
\end{align*}
Note that the boundary ergodic problem 
\begin{align*}
\left\{
\begin{array}{ll}
F(D^2v,Dv,x) &= \lambda \text{~in~} \G \\
L(Dv,x) & = \mu  
\end{array}
\right.
\end{align*}
were studied in  \cite{ERGODIC_PROBLEMS_AND_PERIODIC_HOMOGENIZATION_FOR_FULLY_NONLINEAR_EQUATIONS_IN_HALF_SPACE_TYPE_DOMAINS_WITH_NEUMANN_BOUNDARY_CONDITIONS} by Barles, Da Lio, Lions and Souganidis when $\G$ is a smooth, periodic, half-space-type domain and $F$ a periodic function. They found a constant $\mu$ such that there exists a bounded viscosity solution $v$ of the above problem.

At last we show that we can use the theory of EBSDE to solve an optimal ergodic control problem. $R : U \rightarrow \R^d$ is assumed to be bounded and $L$ is assumed to be Lipschitz and bounded. We define the ergodic cost:
\begin{align}
I(x,\rho) = \limsup_{T \rightarrow + \infty} \frac{1}{T}\E_T^{\rho} \left[ \int_0^T L(X_s^x,\rho_s) \der s \right],
\end{align}
where $\rho$ is an adapted process with values in a separable metric space $U$ and $\E_T^{\rho}$ is the expectation with respect to the probability measure under which $W_t^{\rho} = W_t + \int_0^t R(\rho_s) \der s$ is a Brownian motion on $[0,T]$. Defining 
\begin{align*}
\psi(x,z) = \inf_{u \in U} \{ L(x,u) + zR(u) \}, ~~~~x \in R^d, ~~z\in \R^{1 \times d},
\end{align*}
it is possible to show that, for any admissible control $\rho$, $I(x,\rho) \geq \lambda $. That is why $\lambda$ is called ergodic cost. In a similar way, $\mu$ is called boundary ergodic cost.

The paper is organized as follows. In section $2$, we study the forward SDE under the hypothesis that the drift is weakly dissipative and that the diffusion matrix is invertible and bounded. In this section we prove that the estimates we establish depend on $d$ through its dissipativity coefficient. In section $3$, we use the basic coupling estimate to study existence and uniqueness of an EBSDE with zero Neumann boundary conditions with a forward process weakly dissipative but non-reflected. In section $4$, we use a penalization method to show that the same result holds for a reflected process in a convex set not necessarily bounded. Then, we establish the link between the EBSDE with zero Neumann boundary condition and a PDE. Finally, we apply our results to an optimal ergodic control problem. Some technical proofs are given in the Appendix.

\section{The forward SDE}

\subsection{General notation}
The canonical scalar product on $\R^d$ is denoted by $(~,~)$ and the associated norm is denoted by $|\cdot|$. Given a matrix $\sigma \in \R^{d\times d}$, we define by $||| \cdot |||$ its operator norm. Let $\mathscr{\OO}$ be an open connected subset of $\Rd$. We denote by $\mathscr{C}_b^k(\OO)$ the set of real functions of class $\mathscr{C}^k$ on $\OO$ with bounded partial derivatives. We denote by $\mathscr{C}^k_{\t{lip}}$ the set of real functions whose partial derivatives of order less than or equal to $k$ are Lipschitz. We denote by $B_b(\OO)$ the set of Borel measurable bounded functions defined on $\OO$. 

$(\Omega,\F,\Pb)$ denotes a complete probability space, $(W_t)_{t \geq 0}$ denotes an $\R^d$-valued standard Brownian motion defined on this space and  $(\F_t)_{t \geq 0}$ is the natural filtration of $W$ augmented by $\Pb$-null sets. Then $(\F_t)_{t \geq 0}$ satisfies the usual condition.

$\SSS$ denotes the space of real-valued adapted continuous processes $Y$ such that for all $T > 0$, $\E[\sup_{0 \leq t \leq T} |Y_t|^2 ] < + \infty$.

$\mathscr{S}^p$ denotes the space of real-valued adapted continuous processes $Y$ such that for all $T > 0$, $\E[\sup_{0 \leq t \leq T} |Y_t|^p ] < + \infty$.

$\mathscr{M}^2(\mathbb{R}_+,\mathbb{R}^k)$ denotes the space consisting of all progressively measurable processes $X$, with value in $\mathbb{R}^k$ such that, for all $T>0$,
\[\E\left[ \int_0^T |X_s|^2\der s   \right] < +\infty. \]

 Let $f : \R^d \rightarrow \R^d$ and $\sigma : \R^d \rightarrow \R^{d \times d}$ be two locally Lipschitz functions. We denote by $(V_t^x)_{t \geq 0}$ the strong  solution of the following SDE: 
\begin{equation}\label{SDE avec f}
 V_t^x = x +\int_0^t f(V_s^x)\der s + \int_0^t \sigma(V_s^x) \der W_s.
 \end{equation}

\begin{lemm}\label{lemm estimates dissipativity}
Assume that $\exists a \in \R^d$, $\eta_1$, $\eta_2 > 0$  such that, $\forall y \in \R^d$,
\[(f(y),y-a) \leq -\eta_1 |y-a|^2 + \eta_2, \]
and that $|\sigma|$ is bounded by $\sigma_{\infty}$, then there exists a strong solution $(V_t^x)_{t \geq 0}$ to (\ref{SDE avec f}) which is pathwise unique and for which the explosion time is almost surely equal to infinity. Furthermore the following estimate holds $\forall t \geq 0$:
\[ 
\E|V_t^x|^2 \leq C(1+|x|^2e^{-2\eta_1 t}),
\]
where $C$ is a constant which does not depend on time $t$ and depends on $f$ only through $\eta_1$ and $\eta_2$, and on $\sigma$ only through $\sigma_\infty$.
Furthermore, for all $p > 2$, for all $0 < \beta < p\eta_1$,
\[
\E|V_t^x|^p \leq C(1+|x|^pe^{-\beta t}),
\]
where $C$ is a constant which does not depend on time $t$, depends on $f$ only through $\eta_1$ and $\eta_2$, and on $\sigma$ only through $\sigma_\infty$.
We also have the following inequality, for all $p\geq 1$,
\[\E \left[\sup_{0 \leq t \leq T}|V_t^{x}|^p \right] \leq C(1+|x|^p),
\]
where $C$ depends in this case, on time $T$.
\end{lemm}
\begin{proof}
The proof is given in the appendix.
\end{proof}

We recall that a function is weakly dissipative if it is a sum of an $\eta$-dissipative function (namely $\forall x,x' \in \R^d, (d(x)-d(x') , x-x') \leq -\eta|x-x'|^2$),  and a bounded function. Thus we write $f = d + b$, with $d$ $\eta$-dissipative and $|b|$ bounded by $B$.

\begin{hypo}\label{hypo f weakly dissipative, sigma invertible} ~\\
\vspace{-0.5cm}
\begin{itemize}
\item $f = d + b$ is weakly dissipative, 
\item $d$ is  locally Lipschitz and have polynomial growth: there exists $ \nu > 0$ such that for all $x \in \R^d$, $|d(x)| \leq C(1+|x|^\nu)$,
\item $b$ is Lipschitz,
\item $\sigma$ is Lipschitz, invertible, and $|\sigma|$ and $|\sigma^{-1}|$ are bounded by $\sigma_{\infty}$,
\item there exists $\Lambda \geq 0$ such that for every $x,y \in \R^d$, 
\begin{align*}
|\left(y, \sigma(x+y) - \sigma(x) \right)| \leq \Lambda |y|
\end{align*}
and there exists $\lambda > 0$ such that 
\begin{align*}
2(\lambda - \lambda^2 \Lambda^2) > |||\sigma^{-1}|||^2.
\end{align*}
\end{itemize}
\end{hypo}

\begin{rema}
Note that when $\sigma(x)$ does not depend on $x$, then $\Lambda =0$ and the last assumption of Hypothesis \ref{hypo f weakly dissipative, sigma invertible} is satisfied for $\lambda$ large enough. We give an example of $\sigma$ depending on $x$ and satisfying the last assumption. In the one dimensional case, take $\sigma(x) = 10\mathds{1}_{\{x \leq 0\}} + (10+\frac{1}{10}x)\mathds{1}_{\{0< x < 1\}} + (101/10)\mathds{1}_{\{x \geq 1\}}$. Then $\Lambda  = 1/10$ and $|||\sigma^{-1}||| \leq 1/10$. Clearly, the assumption is satisfied for $\lambda = 1$.
\end{rema}

\begin{rema}\label{rema dissipativity implies previous conditions}
It is clear that if $f$ satisfies Hypothesis \ref{hypo f weakly dissipative, sigma invertible} then $f$ satisfies the assumption of Lemma \ref{lemm estimates dissipativity}. Indeed, let us suppose that $f$ satisfies Hypothesis \ref{hypo f weakly dissipative, sigma invertible}. Let $a \in \R^d$, then $\forall y \in \R^d$,
\begin{align*}
& (f(y)-f(a) , y-a) = (d(y)-d(a),y-a) + (b(y)-b(a),y-a)\\
\Rightarrow & (f(y),y-a) \leq -\eta|y-a|^2 + 2B|y-a| + |f(a)||y-a| \\
\Rightarrow & (f(y),y-a) \leq -\eta|y-a|^2 + \frac{(2B+|f(a)|)^2}{2\varepsilon} + \frac{\varepsilon|y-a|^2}{2},
\end{align*}
which gives us the desired result, for $\varepsilon$ small enough.
\end{rema}

\begin{lemm}\label{lemm basic coupling extended}
Assume that Hypothesis (\ref{hypo f weakly dissipative, sigma invertible}) holds true but this time with $b$ replaced by $b_2$ which is only bounded measurable and not Lipschitz anymore. Then the solution of (\ref{SDE avec f}) with $b$ replaced by $b_2$ still exists but in the weak sense, namely there exists a new Brownian motion $(\hat{W}_t)_{t \geq 0}$ with respect to a new probability measure $\hat{\Pb}$ under which equation (\ref{SDE avec f}) is satisfied by $(V_t^x)_{t \geq 0}$ with $(W_t)_{t \geq 0}$ replaced by $(\hat{W}_t)_{t \geq 0}$. Such a process is unique in law and the estimates of Lemma (\ref{lemm estimates dissipativity}) are still satisfied under the new probability $\hat{\Pb}$.
\end{lemm}
\begin{proof}
It is enough to write:
\begin{align*}
\der V_t^x &= [d(V_t^x) + b(V_t^x)] \der t + \sigma(V_t^x) \der W_t \\
& = [d(V_t^x) +b_2(V_t^x)] \der t  + \sigma(V_t^x)[ \sigma^{-1}(V_t^x)(b(V_t^x) - b_2(V_t^x)) + \der W_t] \\
&= [d(V_t^x) +b_2(V_t^x)] \der t  + \sigma(V_t^x)\der \hat{W}_t,
\end{align*}
where $\der \hat{W}_t = \sigma^{-1}(V_t^x)(b(V_t^x) - b(V_t^x)) + \der W_t$ is the new Brownian motion thanks to the Girsanov theorem (note that $\sigma^{-1}$, $b$ and $b_2$ are measurable and bounded by hypothesis). 
\end{proof}

\begin{lemm}\label{lemm estimates on semi groupe kolmogorov}
Assume that Hypothesis \ref{hypo f weakly dissipative, sigma invertible} holds true. Then there exist $C>0$ and $\mu > 0$ such that $\forall \Phi \in B_b(\R^d)$,
\begin{equation}\label{estimates on kolmogorov semigroup Pt}
|\mathscr{P}_t\left[\Phi\right](x) - \mathscr{P}_t\left[\Phi\right](x') | \leq C(1+|x|^2+|x'|^2)e^{-\mu t}|\Phi|_0
 \end{equation}
where $\mathscr{P}_t[\Phi](x) = \E\Phi(V_t^{x})$ is the Kolmogorov semigroup associated to (\ref{SDE avec f}). We stress the fact that the constants $C$ and $\mu$ depend on $f$ only through $\eta$ and $B$. 
\end{lemm}
\begin{proof}
The proof is given in the appendix.
\end{proof}
\begin{rema}
The importance of the dependency of $C$ and $\mu$ only through some parameters of the problem will appear in Remark \ref{remarque ajout F_n}. 
\end{rema}
\begin{corr}
The estimate (\ref{estimates on kolmogorov semigroup Pt}) can be extended to the case in which $b$ is only bounded measurable and there exists a uniformly bounded sequence of Lipschitz functions $\{b_m\}_{m \geq 1}$ (i.e. $b_m$ is Lipschitz and $\sup_m \sup_x |b_m(x)| < +\infty$) such that 
\[\forall x \in \R^d, ~~ \lim_m b_m(x) = b(x). \]
In this case, we define a semigroup relatively to the new probability measure, namely: 
\begin{align*}
\Pt[\Phi](x) := \hat{\E}\Phi(X_t^x).
\end{align*}
\end{corr}
\begin{proof}
We denote by $\Pt^{m}$ the Kolmogorov semigroup of (\ref{SDE avec f}) with $b$ replaced by $b_m$, for more clarity we rewrite this equation below: $\forall x \in \overline{G}$,
\[
V_t^{x,m} = x + \int_0^t (d +b_m)(V_s^{x,m})\der s + \int_0^t \sigma(V_s^{x,m}) \der W_s.
\]
It is sufficient to prove that, $\forall x \in \overline{G}, ~\forall t \geq 0$,
\[\Pt^{m}[\Phi](x) \rightarrow  \Pt[\Phi](x).\]
To do that, it is easy to adapt the proof from \cite{Debu_Hu_Tess_weak_dissipative} replacing the process $U_t^x$ by its analogue in our context. Thus we define $U_t^x$ as the strong solution of the following SDE:
\[U_t^x = x + \int_0^t d(U_s^x)\der s + \int_0^t \sigma(U_s^x) \der W_s , \]
and the rest remains the same.
\end{proof}


\section{The ergodic BSDE}
In this section we study the following EBSDE in infinite horizon:
\begin{equation}\label{EBSDE général avec f}
Y_t^x = Y_T^x + \int_t^T[ \psi(V_s^x,Z_s^x)- \lambda] \der s - \int_t^T Z_s^x \der W_s, ~ \forall 0 \leq t \leq T < +\infty.
\end{equation}
At the moment, the forward process, defined as the strong solution of (\ref{SDE avec f}) is not reflected. However the existence result we are going to show in the next theorem is interesting for its own, because it gives some ideas which will be reused in the next section. 

 We need the following hypothesis on $\psi : \R^d \times \R^d \rightarrow \R$ : 
\begin{hypo}\label{hypo psi lipschitz en x,z et borné en x}
There exists $M_{\psi} \in \R$ such that: $\forall x \in \R^d$, $\forall z$, $z' \in \R^{1\times d}$, 
\begin{itemize}
\item $\psi : \R^d \times \R^d \rightarrow \R \text{ is measurable},$
\item $\psi(\cdot,z) \text{ is continuous }, $
\item $|\psi(x,0)| \leq M_{\psi},$
\item $|\psi(x,z)-\psi(x,z')| \leq M_\psi |z-z'|. $
\end{itemize}
\end{hypo}

\begin{hypo}\label{hypo f et psi C^1 with partial derivative bounded}~\\
\vspace{-0.5cm}
\begin{itemize} 
\item $f$ is $\mathscr{C}^1$ and all of its derivatives have polynomial growth of first order, i.e. for each $x \in \R^d$ and each multi-index $L$ with $|L| \leq m$, $m \in \left\{0,1\right\}$, there exist positive constants $\gamma_m$ and $q_m$ such that 
\begin{align*}
\left|\partial_L d(x)\right|^2 \leq \gamma_m(1+|x|^{q_m}).
\end{align*}
Also, set $\xi := \max_{m \in \{ 0,1 \}} q_m < + \infty$.
\item $b$, $\sigma$ and $\psi ~ \in ~ \mathscr{C}_b^1$.
\end{itemize}
\end{hypo}

Using a standard approach (see  \cite{FUHRMAN_HU_TESSITORE_ERGODIC_BSDE}), we are going to study the following BSDE in infinite horizon
\begin{equation}\label{BSDE alpha approchant l'EBSDE}
Y_t^{x,\alpha} = Y_T^{x,\alpha} + \int_t^T[ \psi(V_s^{x},Z_s^{x,\alpha})- \alpha Y_s^{x,\alpha}] \der s - \int_t^T Z_s^{x,\alpha} \der W_s, ~ \forall 0 \leq t \leq T < +\infty.
\end{equation}

%

Such an equation was studied in \cite{BRIAND_HU_STABILITY_BSDE_HOMOGENIZATION} from which we have the following result:
\begin{lemm}\label{lemm existence unicité BSDE monotone alpha}
Assume that hypotheses (\ref{hypo f weakly dissipative, sigma invertible}) and (\ref{hypo psi lipschitz en x,z et borné en x}) hold true. Then there exists a unique solution $(Y^{x,\alpha},Z^{x,\alpha})$ to BSDE (\ref{BSDE alpha approchant l'EBSDE}) such that $Y^{x,\alpha}$ is a bounded adapted continuous process and $Z^{x,\alpha} \in \M$. Furthermore, $|Y_t^{x,\alpha}| \leq \frac{M_\psi}{\alpha}$. Finally there exists a function $v^\alpha$ such that $Y_t^{x,\alpha} = v^{\alpha}(X_t^x)$ $\Pb$-a.s.  and there exists a measurable function $\zeta^{\alpha} : \R^d \rightarrow \R^{1 \times d}$ such that $Z_t^{x,\alpha} = \zeta^{\alpha}(X_t^{x})$ $\Pb$-a.s. 
\end{lemm}

We will need the following lemma :

\begin{lemm}\label{lemm zeta approchee}
Let $\zeta$, $\zeta'$ be two continuous functions: $\R^d \rightarrow \R^{1 \times d}$. We define 
\[\Upsilon (x) = \left\{ \begin{array}{l}
\frac{\psi(x,\zeta(x))-\psi(x,\zeta'(x))}{|\zeta(x)-\zeta'(x)|^2}\transpose(\zeta(x)-\zeta'(x)) \t{, if~} \zeta(x) = \zeta'(x),\\
0 \t{, if~} \zeta(x) = \zeta'(x).
\end{array}\right.
\]
There exists a uniformly bounded sequence of Lipschitz functions $(\Upsilon_n)_{n \geq 0}$ (i.e., $\forall n$, $\Upsilon_n$ is Lipschitz and $\sup_n \sup_x  |\Upsilon_n(x)| < +\infty$) such that $\Upsilon_n$ converges pointwisely to $\Upsilon$.
\end{lemm}
\begin{proof}
For all $n \in \N$, we fix infinitely differentiable functions $\rho_n : \R^d \rightarrow \R_+$ bounded together with their derivatives of all order, such that $\int_{\R^d} \rho_n(x) \der x =1$ and 
\begin{align*}
\text{supp}(\rho_n) \subset \left\{ x \in \R^d, |x| \leq \frac{1}{n} \right\},
\end{align*}
where supp denotes the support. Then the required functions $\Upsilon$ can be defined as $\Upsilon_n(x) := \int_{\R^d} \rho_n(x-y)\Upsilon(y) \der y$. 
\end{proof}

The following lemma gives us the desired estimates on $v^{\alpha}(x)$ which will allow us to apply a diagonal procedure.
\begin{lemm}\label{Lemm estimates on v^alpha}
Assume that the hypotheses (\ref{hypo f weakly dissipative, sigma invertible}), (\ref{hypo psi lipschitz en x,z et borné en x}) and (\ref{hypo f et psi C^1 with partial derivative bounded}) hold true. Then, there exists a constant $C>0$ independent of $\alpha$ and which depends on $f$ only through $\eta$ and $B$, on $\sigma$ only through $\sigma_{\infty}$ and on $\psi$ only through $M_{\psi}$ such that, $\forall x,y \in \R^d$,
\begin{align*}
&|v^{\alpha}(x)-v^{\alpha}(y)| \leq C(1+|x|^2+|y|^2),\\
&|v^{\alpha}(x)-v^{\alpha}(y)| \leq C(1+|x|^{2}+|y|^{2})|x-y|. 
\end{align*}
\end{lemm}
\begin{proof}
Let us introduce, as in the paper \cite{WEAK_DIFFERENTIABILITY_OF_SOLUTIONS_SEMI_MONOTONE_DRIFTS}, equation $(3.4)$,  some smooth functions $\Phi_m : \R^d \rightarrow \R$ such that defining
\begin{align*}
d_m(x) = \Phi_m(x) d(x), \forall x \in \R^d,
\end{align*}
$d_m$ is globally Lipschitz, continuously differentiable and for each multi-index $L$ with $|L| \leq  1$,
\begin{align}\label{estimee sur multinindice d_m}
\sup_{m,x} \left\{ | \partial_L \phi_m(x)| + |d(x)\partial_L \phi_m(x)| \right\} \leq C,
\end{align}
for some $C > 0$. We recall that $\phi_m = 1$ on $A_m := \left\{ x \in \R^d; |x|\leq  m^\xi \right\}$. Furthermore,
\begin{align*}
\nabla d_m(x) &= \begin{pmatrix}
\frac{\partial \phi_m(x)}{\partial x_1} d_1(x) & \dots & \frac{\partial \phi_m(x)}{\partial x_d} d_1(x) \\
\vdots & & \vdots \\
 \frac{\partial \phi_m(x)}{\partial x_1} d_d(x) &\dots &  \frac{\partial \phi_m(x)}{\partial x_d} d_d(x) 
\end{pmatrix} + \phi_m(x)\nabla d(x) \\
&= d(x) \nabla \phi_m(x) + \phi_m(x) \nabla d(x).
\end{align*}
Now let us consider $(V_t^{x,m})_{t \geq 0}$ the unique solution of 
\begin{align}\label{SDE lipschitize}
 V_t^{x,m} = x + \int_0^t [d_m(V_s^{x,m}) + b(V_s^{x,m})] \der s + \int_0^t \sigma(V_s^{x,m}) \der W_s, ~~\forall t \geq 0.
\end{align}
We recall that, for each $t \geq 0$ and $p > 1$, $V_t^{x,m}$ converges to $V_t$ in $L^p$ and almost surely.\\
Furthermore the following estimates hold, thanks to the proof of Lemma $3.2$ in \cite{WEAK_DIFFERENTIABILITY_OF_SOLUTIONS_SEMI_MONOTONE_DRIFTS}, for every $p \in \N \setminus \{0,1\}$, there exists $C_p > 0$ such that 
\begin{align}\label{estimates V_t^{x,m}}
\sup_{m \geq 1} \sup_{0 \leq t \leq T} \E \left[ |V_t^{x,m}|^p \right] \leq C_p(1+|x|^{p \nu})e^{C_p T}.
\end{align}
We denote by $\nabla V_s^{x,m} = (\nabla_1 V_s^{x,m} , \dots , \nabla_d V_s^{x,m})$ the solution of the following variational equation (see equation ($2.9$) in \cite{MA_ZHANG_REPRESENTATION_THEOREMS_FOR_BACKWARD_STOCHASTIC_DIFFERENTIAL_EQUATIONS}):
\[\nabla_i V_s^{x,m} = e_i + \int_t^s\nabla (d_m+b)(V_r^{x,m})\nabla_i V_r^{x,m} \der r  + \sum_{j = 1}^d\int_t^s [ \nabla \sigma^j(V_r^{x,m})]\nabla_i V_r^{x,m} \der W_r^j.
\]
Let us mention that the following estimate holds, for every $p>1$:
\begin{align}\label{estimee gradient Vm}
\sup_{m \geq 1} \sup_{0 \leq t \leq T} \E \left[ |\nabla V_t^{x,m}|^p \right] \leq C_{p,T}.
\end{align}

Let us denote by $(Y_t^{x,\alpha,m},Z_t^{x,\alpha,m})_{t \geq 0}$ the unique solution of the following monotone BSDE in infinite horizon, $\forall 0 \leq t \leq T < + \infty$,
\begin{align*}
Y_t^{x,\alpha,m} = Y_T^{x,\alpha,m} + \int_t^T [\psi(V_s^{x,m},Z_s^{x,\alpha,m}) - \alpha Y_s^{x,\alpha,m}] \der s - \int_t^T Z_s^{x,\alpha,m} \der W_s.
\end{align*}
We recall that if we denote by $v^{\alpha,m}$ the following quantity:
\begin{align*}
v^{\alpha,m}(x):=Y_0^{x,\alpha,m}
\end{align*}
then $Y_s^{x,\alpha,m} = v^{\alpha,m}(V_s^x)$.

Now remark that Theorem $4.2$ in \cite{MA_ZHANG_REPRESENTATION_THEOREMS_FOR_BACKWARD_STOCHASTIC_DIFFERENTIAL_EQUATIONS} asks for the terminal condition of the BSDE to be Lipschitz, whereas in our case, the terminal condition $v^{\alpha,m}$ is only continuous and bounded. However, using the same method as that in Theorem $4.2$ in \cite{FUHRMAN_TESSITORE_BISMUT_ELWORTHY_FORMULA_BSDE}, we can readily extend Theorem $4.2$ in \cite{MA_ZHANG_REPRESENTATION_THEOREMS_FOR_BACKWARD_STOCHASTIC_DIFFERENTIAL_EQUATIONS}  for a Markovian terminal condition which is only continuous with polynomial growth, which is our case here.
So, by Theorem $4.2$ in \cite{MA_ZHANG_REPRESENTATION_THEOREMS_FOR_BACKWARD_STOCHASTIC_DIFFERENTIAL_EQUATIONS}, $v^{\alpha,m}$ is continuously differentiable,
\begin{align*}
Z_s^{x,\alpha,m} = \nabla v^{\alpha,m}(V_s^{x,m}) \sigma (V_s^{x,m}), ~~~~ \forall s \in [0,T],~~~~ \Pb\text{-a.s},
\end{align*}
and
\begin{align*}
\nabla v^{\alpha,m}(x) = \E\left[v^{\alpha}(V_T^{x,m})N_T^{0,m} + \int_0^T\left[\psi(V_r^{x,m},Z_r^{x,\alpha,m}) -  \alpha Y_r^{x,\alpha,m} \right]N_r^{0,m} \der r\right],
\end{align*}
where 
\begin{align*}
N_r^{0,m} = \frac{1}{r}\transpose \left(\int_0^r \transpose[\sigma^{-1}(V_s^{x,m})\nabla V_s^{x,m}]\der W_s \right).
\end{align*}
We immediately deduce the following estimate
\begin{align}\label{estimee Nrm}
\E|N_r^{0,m}|^2 &= \frac{1}{|r|^2} \E \int_0^r |\sigma^{-1}(V_s^{x,m}) \nabla V_s^{x,m}|^2 \der s\nonumber \\
&\leq \frac{1}{|r|^2} \sigma_\infty^2 \int_0^r \E |\nabla V_s^{x,m}|^2 \der s\nonumber \\
&\leq \frac{C_{T}}{r}.
\end{align}
As a consequence we can immediately get a uniform bound in $m$ for $\nabla v^{\alpha,m}(x)$. Indeed, 
\begin{align*}
|\nabla v^{\alpha,m}(x)| &\leq \frac{M_\psi}{\alpha}\sqrt{\E\left(\left|N_T^{0,m}\right|^2\right)} + \int_0^T \sqrt{\E\left[(2M_\psi + M_\psi|Z_s^{x,\alpha,m}|)^2\right]}\sqrt{\E\left(\left|N_r^{0,m}\right|^2\right)} \der r \\
&\leq \frac{M_\psi}{\alpha}\frac{C_{T}}{\sqrt{T}} + \int_0^T C\left(1+\sqrt{\E|\nabla v^{\alpha,m}(V_r^{x,m})|^2}\right)\frac{1}{\sqrt{r}}\der r.
\end{align*}
We define $||| \nabla v^{\alpha,m}(x) ||| := \sup_{x \in \R^d} |\nabla v^{\alpha,m}(x)|$, then,
\begin{align*}
|\nabla v^{\alpha,m}(x)| &\leq \frac{C_{T,\alpha}}{\sqrt{T}} + C \sqrt{T} (1+||| \nabla v^{\alpha,m} |||),
\end{align*}
which implies that, taking the supremum over $x$ and for $T$ small enough, for all $x \in \R^d$, 
\begin{align}\label{premiere estimee sur le gradient v{alpha,m}}
|\nabla v^{\alpha,m}(x)| \leq C_{T,\alpha}.
\end{align}

Now we claim that for every $T \geq 0$,
\begin{align*}
\E\left[ |Y_T^{x,\alpha,m} - Y_T^{x,\alpha}|^2 \right] + \E \int_0^T |Z_s^{x,\alpha,m} - Z_s^{x,\alpha}|^2\der s \underset{ m \rightarrow + \infty}{\longrightarrow}0.
\end{align*}
For that purpose, let us denote by $(Y_t^{x,\alpha,n},Z_t^{x,\alpha,n})$ the solution of the following finite horizon BSDE, for all $0 \leq t \leq n$,
\begin{align*}
Y_t^{x,\alpha,n} = 0 + \int_t^{n} \left[\psi(V_s^{x},Z_s^{x,\alpha,n}) - \alpha Y_s^{x,\alpha,n}\right] \der s - \int_t^n Z_s^{x,\alpha,n} \der W_s.
\end{align*}

By inequality ($12$) in \cite{BRIAND_HU_STABILITY_BSDE_HOMOGENIZATION}, $\Pb\text{-a.s.}$, for all $0 \leq t \leq n$,
\begin{align*}
\E|Y_t^{x,\alpha,n}-Y_t^{x,\alpha}|^2 + \E \int_0^t |Z_t^{x,\alpha,n}-Z_t^{x,\alpha}|^2 \der s \leq Ce^{-2\alpha n},
\end{align*}
where $C$ depends only on $M_\psi$ and $\alpha.$
Similarly, let us denote by $(Y_t^{x,\alpha,m,n},Z_t^{x,\alpha,m,n})$ the solution of the following finite horizon BSDE, for all $0 \leq t \leq n$,
\begin{align*}
Y_t^{x,\alpha,m,n} = 0 + \int_t^{n}\left[ \psi(V_s^{x,m},Z_s^{x,\alpha,m,n}) - \alpha Y_s^{x,\alpha,m,n}\right] \der s - \int_t^n Z_s^{x,\alpha,m,n} \der W_s.
\end{align*}
Again, by inequality ($12$) in \cite{BRIAND_HU_STABILITY_BSDE_HOMOGENIZATION}, $\Pb\text{-a.s.}$, for all $0 \leq t \leq n$,
\begin{align*}
\E|Y_t^{x,\alpha,m,,n}-Y_t^{x,\alpha,m}|^2 + \E \int_0^t |Z_t^{x,\alpha,m,n}-Z_t^{x,\alpha,m}|^2 \der s \leq Ce^{-2\alpha n},
\end{align*}
where $C$ depends only on $M_\psi$ and $\alpha$.

Furthermore, thanks to the continuity of $\psi$ in $x$, the following stability result for BSDEs in finite horizon holds (see for example Lemma $2.3$ in \cite{BRIAND_HU_STABILITY_BSDE_HOMOGENIZATION}), for all $0 \leq t \leq n$:
\begin{align*}
\E|Y_t^{x,\alpha,m,n} - Y_t^{x,\alpha,n}|^2 + \E \int_0^n |Z_t^{x,\alpha,m,n} - Z_t^{x,\alpha,n}|^2 \der s \underset{m \rightarrow \infty}{\longrightarrow} 0.
\end{align*}
Then,
\begin{align*}
\E\left[ |Y_T^{x,\alpha,m} - Y_T^{x,\alpha}\right.&\left.|^2 \right] + \E \int_0^T |Z_s^{x,\alpha,m} - Z_s^{x,\alpha}|^2\der s \\
&\leq 3\E\left[ |Y_T^{x,\alpha,m} - Y_T^{x,\alpha,m,n}|^2 \right] + 3\E\left[ |Y_T^{x,\alpha,m,n} - Y_T^{x,\alpha,n}|^2\right]\\
&~~~~~+ 3\E\left[ |Y_T^{x,\alpha,n} - Y_T^{x,\alpha}|^2 \right] + 3\E \int_0^T |Z_s^{x,\alpha,m} - Z_s^{x,\alpha,m,n}|^2\der s \\
&~~~~~+ 3\E \int_0^T |Z_s^{x,\alpha,m,n} - Z_s^{x,\alpha,n}|^2\der s +3\E \int_0^T |Z_s^{x,\alpha,n} - Z_s^{x,\alpha}|^2\der s \\
&\leq Ce^{-2 \alpha n} + 3\E\left[ |Y_T^{x,\alpha,m,n} - Y_T^{x,\alpha,n}|^2\right] + 3\E \int_0^T |Z_s^{x,\alpha,m,n} - Z_s^{x,\alpha,n}|^2\der s.
\end{align*}
Now, for every $\varepsilon > 0$, we pick $n$ large enough such that $2\frac{M_\psi}{\alpha}e^{-\alpha n} < \varepsilon/2$. Then, we choose $m$ large enough such that $3\E\left[ |Y_T^{x,\alpha,m,n} - Y_T^{x,\alpha,n}|^2\right] + 3\E \int_0^T |Z_s^{x,\alpha,m,n} - Z_s^{x,\alpha,n}|^2\der s< \varepsilon/2$. This shows that, for all $x \in \R^d$ and $T\geq 0$,
\begin{align}\label{convergence Ym et Zm vers Y et Z}
\E\left[ |Y_T^{x,\alpha,m} - Y_T^{x,\alpha}|^2 \right] + \E \int_0^T |Z_s^{x,\alpha,m} - Z_s^{x,\alpha}|^2\der s \underset{ m \rightarrow + \infty}{\longrightarrow}0.
\end{align}

In particular taking $T =0$ we deduce that for every $x \in \R^d$,
\begin{align*}
\lim_{m \rightarrow + \infty } v^{\alpha,m}(x) = v^{\alpha}(x).
\end{align*}
Therefore, if we show that $\lim_{m \rightarrow + \infty}\nabla v^{\alpha,m}(x) = h^{\alpha}(x)$ for some function $h^{\alpha}(x)$ then this will imply that $v^{\alpha}$ is continuously differentiable and $\nabla v^{\alpha} = h^{\alpha}$. Furthermore, as $\E \int_0^T |Z_s^{x,\alpha,m} - Z_s^{x,\alpha}|^2\der s \underset{m \rightarrow + \infty}{\longrightarrow} 0$ and $Z_s^{x,\alpha,m} = \nabla v^{\alpha,m}(V_s^{x,m}) \sigma(V_s^{x,m})$, then it will imply that for a.a. $s \geq 0$, $\Pb$-a.s., 
\begin{align}\label{representation pour Z^alpha}
Z_s^{x,\alpha} = \nabla v^{\alpha}(V_s^x)\sigma(V_s^x).
\end{align}

Now we claim that $h^{\alpha}$ can be written as
\begin{align*}
h^{\alpha}(x) = \E\left[v^{\alpha}(V_T^{x})N_T^{0} + \int_0^T\left[\psi(V_r^{x},Z_r^{x,\alpha}) -  \alpha Y_r^{x,\alpha} \right]N_r^{0} \der r\right],
\end{align*}
where 
\begin{align*}
N_r^{0} = \frac{1}{r}\transpose \left(\int_0^r \transpose[\sigma^{-1}(V_s^{x})\nabla V_s^{x}]\der W_s \right).
\end{align*}

Indeed, for every $x \in \R^d$,
\begin{align*}
|\nabla &v^{\alpha,m}(x)-h^{\alpha}(x)|\\
&\leq \E|v^{\alpha,m}(V_T^{x,m})N_T^{0,m} - v^{\alpha}(V_T^{x})N_T^0| \\
&~~~~+ \E\int_0^T |(\psi(V_r^{x,m},Z_r^{x,\alpha,m}) - \alpha Y_r^{x,\alpha,m})N_r^{0,m} - (\psi(V_r^{x},Z_r^{x,\alpha})-\alpha Y_r^{x,\alpha})N_r^0|\der r\\
&\leq \E|v^{\alpha,m}(V_T^{x,m})N_T^{0,m}-v^{\alpha,m}(V_T^{x,m})N_T^{0}| + \E|v^{\alpha,m}(V_T^{x,m})N_T^0-v^{\alpha}(V_T^{x})N_T^0|\\
&~~~~+ \E\int_0^T |(\psi(V_r^{x,m},Z_r^{x,\alpha,m})-\alpha Y_r^{x,\alpha,m})N_r^{0,m}-(\psi(V_r^{x,m},Z_r^{x,\alpha,m})-\alpha Y_r^{x,\alpha,m})N_r^{0}|\der r\\
&~~~~+ \E\int_0^T|(\psi(V_r^{x,m},Z_r^{x,\alpha,m})-\alpha Y_r^{x,\alpha,m})N_r^0-(\psi(V_r^{x},Z_r^{x,\alpha})-\alpha Y_r^{x,\alpha})N_r^0|\der r\\
&\leq \frac{M_\psi}{\alpha}\sqrt{\E(|N_T^{0,m}-N_T^0|^2)} + \sqrt{\E(|N_T^0|^2)}\sqrt{\E(|v^{\alpha,m}(V_T^{x,m})-v^{\alpha}(V_T^{x})|^2)}\\
&~~~~+ \int_0^T \sqrt{C(1+\E(|\nabla v^{\alpha,m}(V_s^{x,m})|^2))}\sqrt{\E(|N_r^{0,m}-N_r^{0}|^2)}\der r\\
&~~~~+ \int_0^T\sqrt{\E(|N_r^0|^2)} \sqrt{\E(|\psi(V_r^{x,m},Z_r^{x,\alpha,m})-\alpha Y_r^{x,\alpha,m}-\psi(V_r^{x},Z_r^{x,\alpha})+\alpha Y_r^{x,\alpha}|^2)} \der r\\
&\leq \frac{M_\psi}{\alpha}\sqrt{\E(|N_T^{0,m}-N_T^0|^2)} + \sqrt{\E(|N_T^0|^2)}\sqrt{\E(|v^{\alpha,m}(V_T^{x,m})-v^{\alpha}(V_T^{x})|^2)}\\
&~~~~+ C_{T,\alpha}(\sqrt{T}+\frac{1}{\sqrt{T}})\int_0^T \sqrt{\E(|N_r^{0,m}-N_r^{0}|^2)} \der r\\
&~~~~+ \int_0^T \frac{C_{T}}{\sqrt{r}} C_{T,\alpha}\sqrt{\E(|Z_r^{x,\alpha,m}-Z_r^{x,\alpha}|^2)}\der r\\
&~~~~+ \sqrt{\E(|\psi(V_r^{x,m},Z_r^{x,\alpha})-\psi(V_r^{x},Z_r^{x,\alpha})|^2)} \der r \\
&~~~~+\int_0^T \frac{C_T}{\sqrt{r}} \sqrt{\E(|Y_r^{x,\alpha,m}-Y_r^{x,\alpha}|^2)}  \der r
\end{align*}
where we have used the estimate (\ref{premiere estimee sur le gradient v{alpha,m}}) for the last inequality.

We have
\begin{align*}
\E(|N_T^{0,m}-N_T^0|^2) = \frac{1}{T^2}\int_0^T \E(|\sigma^{-1}\nabla V_s^{x,m} - \sigma^{-1}\nabla V_s^x|^2) \der s \underset{m \rightarrow + \infty}{\longrightarrow} 0,
\end{align*}
since $\sigma^{-1}(V_s^{x,m})\nabla V_s^{x,m} - \sigma^{-1}(V_s^x)\nabla V_s^x \underset{m \rightarrow + \infty}{\longrightarrow} 0 $ $\Pb$-a.s. and since
\begin{align*}
\sup_m \E(|\sigma^{-1}(V_s^{x,m})\nabla V_s^{x,m} - \sigma^{-1}(V_s^x)\nabla V_s^x|^4) < +\infty
\end{align*}
by estimate (\ref{estimee gradient Vm}).

The second and the third term in the sum converge toward $0$ by the dominated convergence theorem. The fourth one converges toward $0$ by Jensen's inequality and the dominated convergence theorem and the last two ones converge toward $0$ by the dominated convergence theorem.\\
Now the first estimate of the lemma can be established exactly as in Lemma $3.6$ of \cite{Debu_Hu_Tess_weak_dissipative} thanks to the representation formula (\ref{representation pour Z^alpha}). \\
Let us establish the second inequality of the lemma. We have, using the following notation $\overline{v}^{\alpha}(x) = v^{\alpha}(x)-v^{\alpha}(0)$,
\begin{align}\label{applique formule représentation MA ZHANG}
|\nabla \overline{v}^{\alpha}(x)| = \left| \E\left[\overline{v}^{\alpha}(V_T^x)N_T^{0} + \int_0^T  \left[\psi(V_s^x,Z_s^x) - \alpha \overline{v}^{\alpha}(V_r^x) - \alpha v^{\alpha}(0)\right]N_r^0\der r  \right] \right|.
\end{align}
We have, using the first inequality of the lemma and inequality (\ref{estimee Nrm}):
\begin{align*}
\E|\overline{v}^{\alpha}(V_T^{x})N_T^{0}| \leq C\frac{(1+|x|^2)}{\sqrt{T}}.
\end{align*}
Furthermore, since we can assume that $\alpha \leq 1$:
\begin{align*}
\E \int_0^T &\left|\left[\psi(V_r^{x},Z_r^{x,\alpha}) -  \alpha \overline{Y}_r^{x,\alpha}- \alpha v^{\alpha}(0) \right]N_r^{0} \right| \der r \\
&\leq C\E \int_0^T \left(M_\psi + M_\psi|Z_r^{x,\alpha}| + C(1+|V_r^{x}|^2) + M_\psi \right)|N_r^{0}| \der r \\
& \leq C\E\int_0^T |N_r^{0}| \der r + C\E\int_0^T |Z_r^{x,\alpha}||N_r^{0}| \der r + C\E\int_0^T (1+|V_r^{x}|^2)|N_r^{0}| \der r  \\
& =: I_1 + I_2 + I_3.
\end{align*}
We easily get $I_1 \leq C$. Furthermore, thanks to the representation formula (\ref{representation pour Z^alpha}) and the fact that $|\sigma|$ is bounded:
\begin{align*}
I_2 \leq C\int_0^T \sqrt{\E|\nabla v^{\alpha}(V_s^{x})|^2}\frac{1}{\sqrt{r}} \der r.
\end{align*}
Regarding $I_3$, we easily get $I_3 \leq C(1+|x|^2)$.\\
Now we define $|||\nabla v^{\alpha} ||| := \sup_{x \in \R^d} \frac{|\nabla v^{\alpha}(x)|}{1+|x|^2}$, then coming back to equation (\ref{applique formule représentation MA ZHANG}), we have
\begin{align*}
|\nabla v^{\alpha}(x)| &\leq C\left(1 +|x|^2 + \frac{1+|x|^2}{\sqrt{T}}\right) + C\int_0^T \sqrt{\E(1+|V_s^{x}|^2)^2} ||| \nabla v^{\alpha} |||\frac{1}{\sqrt{r}} \der r \\
&\leq C\left(1 +|x|^2 + \frac{1+|x|^2}{\sqrt{T}}\right) + C\int_0^T (1+|x|^2) ||| \nabla v^{\alpha} |||\frac{1}{\sqrt{r}} \der r.
\end{align*}
This implies
\begin{align*}
|||\nabla v^{\alpha}||| \leq C\left(1+\frac{1}{\sqrt{T}}\right) + C \sqrt{T}|||\nabla v^{\alpha} ||||.
\end{align*}
Thus, for $T$ small enough:
\begin{align*}
|||\nabla v^{\alpha} ||| \leq C\left(1+\frac{1 + \sqrt{T}}{\sqrt{T}(1-C\sqrt{T})}\right),
\end{align*}
which implies that, for all $x \in \R^d$,
\begin{align*}
|\nabla v^{\alpha}(x)| \leq C(1+|x|^2).
\end{align*}
This last estimate gives us, for all $x, y \in \R^d$,
\begin{align*}
|v^{\alpha}(x) - v^{\alpha}(y)| \leq C(1+|x|^2+|y|^2)|x-y|.
\end{align*}
\end{proof}


Thanks to this estimate, it is possible to get an existence result for EBSDE (\ref{EBSDE général avec f}). Here Hypothesis $\ref{hypo f et psi C^1 with partial derivative bounded}$ can be removed thanks to a convolution argument which will appear in the proof.

\begin{théo}\label{théorème existence EDSRE cas général avec f dissipatif}
Assume that the hypotheses (\ref{hypo f weakly dissipative, sigma invertible}) and (\ref{hypo psi lipschitz en x,z et borné en x}) hold true. Then there exists a solution $(\overline{Y}^{x},\overline{Z}^{x},\overline{\lambda})$ to EBSDE (\ref{EBSDE général avec f}) such that $\overline{Y}_\cdot^x=\overline{v}(V_\cdot^x)$ with $\overline{v}$ locally Lipschitz, and there exists a measurable function $\overline{\xi} : \R^{d} \rightarrow \R^{1 \times d}$ such that $\overline{Z}^x \in \M$ and $\overline{Z}_\cdot^x = \overline{\xi}(V_\cdot^x)$.
 \end{théo}
 \begin{proof}
 We start by regularizing $f$ and $\psi$ thanks to classical convolution arguments. For all $k \in \N^*$ let us denote by $\rho_{\varepsilon}^k : \R^k \rightarrow \R_+$ the classical mollifier for which the support is the ball of center $0$ and radius $\varepsilon$. Let us denote for a sequence $(\varepsilon_n)_{n \in \N} \in \R_+$  such that $\varepsilon_n \underset{n \rightarrow + \infty}{\longrightarrow} 0$, $d^{\varepsilon_n} := d*\rho_{\varepsilon_n}^d$, $b^{\varepsilon_n} := b*\rho_{\varepsilon_n}^d$,  $\psi^{\varepsilon_n} := \psi*\rho_{\varepsilon_n}^{d,d}$ and $\sigma^{\varepsilon_n} := \sigma*\rho_{\varepsilon_n}^{d \times d}$. Those functions are $\mathscr{C}^1$ and satisfies:
 \begin{itemize}
 \item $d^{\varepsilon_n}$ is $\eta$-dissipative;
 \item $|d^{\varepsilon_n}(x)| \leq C(1+|x|)^p$, for a $p \geq 0$;
 \item $|\nabla d^{\varepsilon}(x)| \leq C_\varepsilon(1+|x|^q)$, for a $q \geq 0$;
 \item $b^{\varepsilon_n}$ is bounded  by $B$;
 \item $\psi^{\varepsilon_n}$  satisfies Hypothesis \ref{hypo psi lipschitz en x,z et borné en x};
 \item $\sigma^{\varepsilon_n} $ is invertible;
 \item $d^{\varepsilon_n}\rightarrow d$, $b^{\varepsilon_n} \rightarrow b$, $\psi^{\varepsilon_n} \rightarrow \psi$, $\sigma^{\varepsilon_n} \rightarrow \sigma$ pointwisely as $\varepsilon_n \rightarrow 0$.
 \end{itemize}
Note now that Hypothesis \ref{hypo f et psi C^1 with partial derivative bounded} is satisfied by the regularized functions defined above, therefore Lemma \ref{Lemm estimates on v^alpha} can be applied. 
 We just precise that the pointwise convergence of the regularized functions is a consequence of the continuity of the functions $d$, $b$, $\psi$ and $\sigma$. 
 We denote by $V_t^{x,\varepsilon_n}$ the solution of  (\ref{SDE avec f}) with $f$ replaced by $f^{\varepsilon_n}$ and $\sigma$ replaced by $\sigma^{\varepsilon_n}$. The same notation is used for the regularized BSDE, we denote by  $(Y_t^{x,\alpha,\varepsilon_n}, Z_t^{x,\alpha,\varepsilon_n})$ the solution in $\SSS \times \M$ of BSDE (\ref{BSDE alpha approchant l'EBSDE}) with $\psi$ replaced by $\psi^{\varepsilon_n}$ (existence and uniqueness of such a solution is guaranteed by Lemma \ref{lemm existence unicité BSDE monotone alpha}), namely  $ \forall 0\leq t\leq T < +\infty$:
\begin{align}\label{EDSR_avec_generateur_monotone_et_drift_forward_penalise}
Y_t^{x,\alpha,\varepsilon_n} = &Y_T^{x,\alpha,\varepsilon_n} + \int_t^T (\psi(V_s^{x,\varepsilon_n},Z_s^{x,\alpha,\varepsilon_n}) - \alpha Y_s^{x,\alpha,\varepsilon_n}) \der s - \int_t^T Z_s^{x,\alpha,\varepsilon_n} \der W_s.
\end{align}
Then we define $v^{\alpha,\varepsilon_n}(x) := Y_0^{x,\alpha,\varepsilon_n}$ and $\overline{Y}_t^{x,\alpha,\varepsilon_n} = Y_t^{x,\alpha,\varepsilon_n} - \alpha v^{\alpha,\varepsilon_n}(0)$. We can rewrite the BSDE and we get:
\begin{align*}
\overline{Y}_t^{x,\alpha,\varepsilon_n} = &\overline{Y}_T^{x,\alpha,\varepsilon_n} + \int_t^T (\psi(V_s^{x,\varepsilon_n},Z_s^{x,\alpha,\varepsilon_n}) - \alpha \overline{Y}_s^{x,\alpha,\varepsilon_n} - \alpha v^{\alpha,\varepsilon_n}(0)) \der s\\
&- \int_t^T Z_s^{x,\alpha,\varepsilon_n} \der W_s, ~~~0\leq t\leq T < +\infty .
\end{align*}

Uniqueness of solutions implies that $v^{\alpha,\varepsilon_n}(V_s^{x,\varepsilon_n}) = Y_s^{x,\alpha,\varepsilon_n}$. Now, in a very classical way,  we set $\overline{v}^{\alpha,\varepsilon_n}(x) = v^{\alpha,\varepsilon_n}(x)-v^{\alpha,\varepsilon_n}(0)$. Thanks to the fact that $\alpha |v^{\alpha,\varepsilon_n}(0)| \leq M_\psi$ and by Lemma \ref{Lemm estimates on v^alpha} we can extract a subsequence $\beta(\varepsilon_n) \underset{n \rightarrow + \infty}{\rightarrow} 0$ such that $\forall \alpha > 0$, $\forall x \in D$ a countable subset of $\R^d$: 
\[ \overline{v}^{\alpha,\beta(\varepsilon_n)}(x) \underset{n \rightarrow + \infty}{\longrightarrow} \overline{v}^{\alpha}(x)    \t{~~~and~~~} \alpha v^{\alpha,\beta(\varepsilon_n)}(0) \underset{n \rightarrow + \infty}{\longrightarrow} \overline{\lambda}^{\alpha},\]
for a suitable function $\overline{v}$ and a suitable real $\overline{\lambda}^{\alpha}$. Now thanks to the estimates from Lemma \ref{Lemm estimates on v^alpha} we have $\forall \alpha > 0$, $|\overline{v}^{\alpha,\beta(\varepsilon_n)}(x)-\overline{v}^{\alpha,\beta(\varepsilon_n)}(x')| \leq c(1+|x|^2+|x'|^2)|x-x'|$  for all $x,x' \in \R^d$. Therefore extending $\overline{v}^{\alpha}$ to the whole $\R^d$ by setting $\overline{v}^{\alpha}(x) = \lim_{x_p \rightarrow x}\overline{v}^{\alpha}(x_p)$ we still have the following estimates: for all $x, x' \in \R^d$,
\[|\overline{v}^{\alpha}(x)-\overline{v}^{\alpha}(x')| \leq C(1+|x|^2+|x'|^2)|x-x'|.\]
In addition, we also have 
\begin{align*}
|\overline{\lambda}^\alpha| \leq M_\psi.
\end{align*}
Now let us define $\forall t \geq 0$, $\overline{Y}_t^{x,\alpha} = \overline{v}^{\alpha}(V_t^x)$. Let us show that 
 \[\E\int_0^T|\overline{Y}_s^{x,\alpha,\beta(\varepsilon_n)} - \overline{Y}_s^{x,\alpha}|^2\der s \underset{n \rightarrow + \infty}{\rightarrow} 0 \t{~~and~~}  \E|\overline{Y}_T^{x,\alpha,\beta(\varepsilon_n)} - \overline{Y}_T^{x,\alpha}|^2 \underset{n \rightarrow + \infty}{\rightarrow} 0.\]
First we write: 
\begin{align*}
|\overline{v}^{\alpha,\beta(\varepsilon_n)}(V_s^{x,\beta(\varepsilon_n)}) - \overline{v}^{\alpha}(V_s^x)| &\leq |\overline{v}^{\alpha,\beta(\varepsilon_n)}(V_s^{x,\beta(\varepsilon_n)})-\overline{v}^{\alpha,\beta(\varepsilon_n)}(V_s^x)| \\
&~~~~~~~~~~~~~~~~~~~~~~~~~~~~~~~~~~+ |\overline{v}^{\alpha,\beta(\varepsilon_n)}(V_s^x) - \overline{v}^{\alpha}(V_s^x)|\\
& \leq  C(1+|V_s^{x,\beta(\varepsilon_n)}|^2+|V_s^x|^2)|V_s^{x,\beta(\varepsilon_n)}-V_s^x| \\
&~~~~~~~~~~~~~~~~~~~~~~~~~~~~~~~~~~ + |\overline{v}^{\alpha,\beta(\varepsilon_n)}(V_s^x) - \overline{v}^{\alpha}(V_s^x)|,
\end{align*}
which shows the convergence of $\overline{v}^{\alpha,\beta(\varepsilon_n)}(V_s^{x,\beta(\varepsilon_n)})$ toward $\overline{v}^{\alpha}(V_s^x)$ almost surely, up to a subsequence (it is well known that $\forall T > 0$, $E \sup_{0 \leq t \leq T}|V_t^{x,\beta(\varepsilon_n)} - V_t^x|^2 \underset{n \rightarrow + \infty}{\longrightarrow} 0$).
Then, due to the fact that \\ $|\overline{v}^{\alpha,\beta(\varepsilon_n)}(V_s^{x,\beta(\varepsilon_n)})| \leq M_\psi/\alpha$ $\Pb$-a.s., we can apply the dominated convergence theorem to show that:
\[\E\int_0^T|\overline{Y}_s^{x,\alpha,\beta(\varepsilon_n)} - \overline{Y}_s^{x,\alpha}|^2\der s \underset{n \rightarrow + \infty}{\rightarrow} 0 \t{~~and~~}  \E|\overline{Y}_T^{x,\alpha,\beta(\varepsilon_n)} - \overline{Y}_T^{x,\alpha}|^2 \underset{n \rightarrow + \infty}{\rightarrow} 0.\]

Now we show that $(Z^{x,\alpha,\beta(\varepsilon_n)})_{n}$ is Cauchy in $\M$. 
 We denote 
\begin{align*}
\widetilde{V}_t=V_t^{x,\beta(\varepsilon_n)}-V_t^{x,\beta(\varepsilon_n)'};
\end{align*}
 \begin{align*}
 \widetilde{Y}_t=\overline{Y}_t^{x,\alpha,\beta(\varepsilon_n)}-\overline{Y}_t^{x,\alpha,\beta(\varepsilon_n)'};
\end{align*} 
 \begin{align*}
  \widetilde{Z}_t=\overline{Z}_t^{x,\alpha,\beta(\varepsilon_n)}-\overline{Z}_t^{x,\alpha,\beta(\varepsilon_n)'};
 \end{align*}
 and 
\begin{align*}
 \widetilde{\lambda}=\alpha v^{\alpha,\beta(\varepsilon_n)}(0) - \alpha v^{\alpha,\beta(\varepsilon_n)'}(0). 
\end{align*} 

Itô's formula applied to $|\widetilde{Y}_t|^2$ gives us, for all $\varepsilon_1, \varepsilon_2, \varepsilon_3 > 0$:
\begin{align*}
|\widetilde{Y}_t|^2 + \int_t^T|\widetilde{Z}_s|^2\der s &= |\widetilde{Y}_T|^2 +  2\int_t^T\widetilde{Y}_t[\psi(V_s^{x,\beta(\varepsilon_n)},Z_s^{x,\alpha,\beta(\varepsilon_n)})-\psi(V_s^{x,\beta(\varepsilon_n)'},Z_s^{x,\alpha,\beta(\varepsilon_n)'}) \\
&~~~~ - (\alpha \overline{Y}_s^{x,\alpha,\beta(\varepsilon_n)}-\alpha \overline{Y}_s^{x,\alpha,\beta(\varepsilon_n)'}) -\widetilde{\lambda}]\der s \\
&~~~~ -2\int_t^T\widetilde{Y}_s\widetilde{Z}_s\der W_s\\
& \leq |\widetilde{Y}_T|^2 +  (\varepsilon_1M_\psi+ \varepsilon_2M_\psi+\varepsilon_3)\int_t^T|\widetilde{Y}_s|^2 \der s + \frac{M_\psi}{\varepsilon_1}\int_t^T|\widetilde{V}_s|^2\der s \\
&~~~~+  \frac{M_\psi}{\varepsilon_2}\int_t^T|\widetilde{Z}_s|^2\der s +  \frac{1}{\varepsilon_3}\int_t^T|\widetilde{\lambda}|^2\der s +c\int_t^T|\widetilde{Y}_s|\der s - 2\int_t^T\widetilde{Y}_s\widetilde{Z}_s \der W_s,
\end{align*}
because $\alpha |v^{\alpha,\varepsilon}(0)| \leq M_\psi $. Thus, taking the expectation and for $\varepsilon_2$ large enough we get 
\[\E\int_0^T|\widetilde{Z}_s|^2 \der s \leq \E|\widetilde{Y}_T|^2 + c\left(\E\left[\int_0^T|\widetilde{Y}_s|^2\der s\right] + \E\left[\int_0^T|\widetilde{V_s}|^2 \der s\right] + \E\left[\int_0^T|\widetilde{Y}_s|\der s\right]+  T|\widetilde{\lambda}|^2 \right),\]
which proves that $(Z^{x,\alpha,\beta(\varepsilon_n)})_{\beta(\varepsilon_n)}$ is Cauchy in $\M$. 
Now we pass to the limit in equation (\ref{EDSR_avec_generateur_monotone_et_drift_forward_penalise}) to obtain:
\begin{align*}
Y_t^{x,\alpha} = &Y_T^{x,\alpha} + \int_t^T (\psi(V_s^{x},Z_s^{x,\alpha}) - \overline{\lambda}^{\alpha}) \der s - \int_t^T Z_s^{x,\alpha} \der W_s, ~~~0\leq t\leq T < +\infty .
\end{align*}

Now we reiterate the above method. Thanks to the following estimates: 
 $\forall x, x' \in \R^d$,
\[|\overline{v}^{\alpha}(x)-\overline{v}^{\alpha}(x')| \leq C(1+|x|^2+|x'|^2)|x-x'|,\]
and  
\begin{align*}
|\overline{\lambda}^\alpha| \leq M_\psi,
\end{align*}
it is possible, by a diagonal procedure, to construct a sequence $(\alpha_n)_n$ such that 
\begin{align*}
&\overline{v}^{\alpha_n}(x) \underset{n \rightarrow + \infty}{ \longrightarrow } \overline{v}(x)\\
&\overline{\lambda}^{\alpha_n} \underset{n \rightarrow + \infty}{ \longrightarrow } \overline{\lambda}.
\end{align*}
We define $\overline{Y}_t^x := \overline{v}(V_t^x)$. Let us just precise why
 \[\E\int_0^T|\overline{Y}_s^{x,\alpha} - \overline{Y}_s^{x}|^2\der s \underset{n \rightarrow + \infty}{\rightarrow} 0 \t{~~and~~}  \E|\overline{Y}_T^{x,\alpha} - \overline{Y}_T^{x}|^2 \underset{n \rightarrow + \infty}{\rightarrow} 0.\]

First the convergence of $\overline{v}^{\alpha_n}(V_s^x)$ toward $\overline{v}(V_s^x)$ is clear. Secondly, we have 
\begin{align*}
|{v}^{\alpha_n}(V_s^x)| \leq C(1+|V_s^x|^2) .
\end{align*}
Therefore the dominated convergence theorem can be applied to show that:
\[\E\int_0^T|\overline{Y}_s^{x,\alpha,\beta(\varepsilon_n)} - \overline{Y}_s^{x,\alpha}|^2\der s \underset{n \rightarrow + \infty}{\rightarrow} 0 \t{~~and~~}  \E|\overline{Y}_T^{x,\alpha,\beta(\varepsilon_n)} - \overline{Y}_T^{x,\alpha}|^2 \underset{n \rightarrow + \infty}{\rightarrow} 0.\]
Then, just as before, it is possible to show that $(Z^{x,\alpha_n})_{\alpha_n}$ is Cauchy in $\M$. We denote its limit by $\overline{Z}_s^x$.

The end of the proof is very classical, it suffices to apply BDG's inequality to show that $\E \sup_{0 \leq t \leq T}|\overline{Y}^x|^2 < + \infty$ , $\forall T > 0$.
To show that $\overline{Z}^x$ is Markovian, just apply the same method as in the proof of Theorem $4.4$ in \cite{FUHRMAN_HU_TESSITORE_ERGODIC_BSDE}.
\end{proof}
\begin{rema}
It is clear that we do not have uniqueness of the solutions of EBSDE (\ref{EBSDE général avec f}) because if $(Y,Z,\lambda)$ is a solution then $(Y + \theta,Z,\lambda)$ is another solution, for all $\theta \in \R$. However we have a uniqueness property for $\lambda$ under the following polynomial growth property: 
\begin{align*}
|Y_t^x| \leq C(1+|V_t^x|^2).
\end{align*}
One can notice that the solution $\overline{Y}_t^x = \overline{v}(V_t^x)$ constructed in the proof of Theorem (\ref{théorème existence EDSRE cas général avec f dissipatif}) satisfies such a growth property.
\end{rema}
\begin{théo}
(Uniqueness of $\lambda$). 
Assume that the hypotheses (\ref{hypo f weakly dissipative, sigma invertible}) and (\ref{hypo psi lipschitz en x,z et borné en x}) hold true. Let us suppose that we have two solutions of EBSDE (\ref{EBSDE with zero neumann condition}) denoted by $(Y,Z,\lambda)$ and $(Y',Z',\lambda')$ where $Y$ and $Y'$ are progressively measurable continuous processes, $Z$ and $Z'$ $\in \M$ and $\lambda$, $\lambda' \in \R$. Finally assume that the following growth properties hold:
\begin{align*}
&|Y_t| \leq C(1+|V_t^x|^2)\\
&|{Y}_t'| \leq C'(1+|V_t^x|^2).
\end{align*}
Then $\lambda = \lambda'$.
\end{théo}
\begin{proof}
It suffices to adapt the proof of Theorem 4.6 of \cite{FUHRMAN_HU_TESSITORE_ERGODIC_BSDE}. With the same notations one can write: 
\begin{align*}
\widetilde{\lambda} &= T^{-1} \E^{\Pb_h}[\widetilde{Y}_T - \widetilde{Y}_0] \\
&\leq T^{-1}\E^{\Pb_h}((C+C')(1+|V_T^x|^2)) + T^{-1}\E^{\Pb_h}((C+C')(1+|x|^2)). 
\end{align*}
To conclude, just use the estimates from Lemma \ref{lemm estimates dissipativity}, and let $T \rightarrow + \infty$.
\end{proof}


\section{The ergodic BSDE with zero and non-zero Neumann\\ boundary conditions in a weakly dissipative environment}
In this section we replace the process $(V_t^x)_{t \geq 0}$ by the process $(X_t^{x})_{t \geq 0}$,
which is solution of  a stochastic differential equation reflected in the closure of an open convex subset $G$ of $\R^d$ with regular boundary, namely, we consider the  following stochastic equation for a pair of unknown processes $(X_t^x,K_t^x)_{t \geq 0}$ such that, for every $x \in \overline{G}$, $t \geq 0$ :

\begin{equation}\label{SDE reflected}
  \left\{
      \begin{aligned}
&X_t^x = x + \int_0^t f(X_s^x)\der s + \int_0^t \sigma(X_s) \der s + \int_0^t \nabla \phi (X_s^x)\der K_s^x,\\
&K_t^x = \int_0^t \mathds{1}_{\{X_s^x \in \partial G \}} \der K_s^x,
      \end{aligned}
    \right.
\end{equation}
where $f$ is weakly dissipative.

As far as we know, there is no result regarding such diffusions. That is why it is necessary to adapt a result of Menaldi in \cite{MENALDI_STOCHASTIC_VARIATIONNAL_INEQUALITY_FOR_REFLECTED_DIFFUSION} where an existence and uniqueness result is stated by a penalization method for a diffusion reflected in a convex and bounded set under Lipschitz assumptions for the drift. Therefore, it is necessary to adapt this result in our framework, namely when the set is not bounded anymore but with weakly dissipative assumptions for the drift.

We denote by $\Pi(x)$ the projection of $x \in \R^d$ on $\overline{G}$. Let us denote by $(X_t^{x,n})_{t \geq 0}$ the unique strong solution of the following penalized problem associated to the reflected problem (\ref{SDE reflected}) :
\begin{equation}\label{SDE pénalisée}
X_t^{x,n} = x + \int_0^t (d + F_n +b)(X_s^{x,n})\der s + \int_0^t \sigma(X_s^{x,n}) \der W_s,
\end{equation}
where $\forall x \in \R^d, ~~ F_n(x) = -2n(x-\Pi(x)). $

\begin{rema}\label{remarque ajout F_n}
The  functions $d+F_n+b$ and $\sigma$ satisfy Hypothesis \ref{hypo f weakly dissipative, sigma invertible}. Indeed, from \cite{GEGOUT-PETIT_PARDOUX_EDSR_REFLECHIES_DANS_UN_CONVEXE}, $F_n$ is $0$-dissipative therefore $d+F_n$ remains $\eta$-dissipative thus the estimate of Lemma \ref{lemm estimates on semi groupe kolmogorov} holds with constants which do not depend on $n$. Furthermore one can remark that for all $\xi \in \R^d$, $\transpose \xi  \nabla F_n(x) \xi \leq 0$, for all $x \in \R^d$ (see for example \cite{GEGOUT-PETIT_PARDOUX_EDSR_REFLECHIES_DANS_UN_CONVEXE}). Finally, taking $a \in \G$ (thus $F_n(a) = 0$) in Remark \ref{rema dissipativity implies previous conditions} shows us that the estimate of Lemma \ref{lemm estimates dissipativity} holds with constants that do not depend on $n$. 
\end{rema}

We need the following assumptions on $G$: 
\begin{hypo}\label{hypo G convexe}
$G$ is an open convex set of $\R^d$.
\end{hypo}

\begin{hypo}\label{hypo bord de G}
There exists a function $\phi \in \mathscr{C}_b^2(\mathbb{R}^d)$ such that $G = \{\phi > 0\}$, $\partial G = \{\phi = 0 \}$ and $|\nabla \phi (x) | = 1$, $\forall x \in \partial G$.
\end{hypo}

The following Lemma states that the penalized process is Cauchy in the space of predictable continuous process for the norm $\E \sup_{0 \leq t \leq T} | \cdot |^p$, for every $ p > 2$ and that it converges to the reflected process solution of (\ref{SDE reflected}) for a  process $K^x$ with bounded variations.
\begin{lemm}\label{lemm convergence processus prenalise}
Assume that the hypotheses (\ref{hypo f weakly dissipative, sigma invertible}), (\ref{hypo G convexe}) and (\ref{hypo bord de G}) hold true. Then for every $x \in \overline{G}$, there exists a unique pair of processes $\left\{(X_t^x,K_t^x)_{t \geq 0} \right\}  $ with values in $(\G \times \R_+)$ and which belong to the space $\mathscr{S}^p \times \mathscr{S}^p$, $\forall 1 \leq p < + \infty$, satisfying (\ref{SDE reflected}) and such that 
\begin{align*}
\eta_t^x := \int_0^t \nabla \phi(X_s^x) \der K_s^x ~~~~~\text{has bounded variation on $[0,T]$, $0 < T < \infty$, $\eta^x_0 =0$}
\end{align*}
and for all process $z$ continuous and progressively measurable taking values in the closure $\overline{G}$ we have
\begin{align*}
\int_0^T (X_s^x-z_s) \der \eta^x_s \leq 0, ~~~~~\forall T > 0.
\end{align*}
Finally the following estimate hold for the convergence of the penalized process, for any $1 < q < p/2$,  for any $T \geq 0$ there exists $C>0$ such that
\begin{align*}
\E \sup_{0 \leq t \leq T} |X_t^{x,n}-X_t^x|^p \leq C\left(\frac{1}{n^q}\right),
\end{align*}
\end{lemm}
\begin{proof}
The proof is given in Appendix.
\end{proof}

\subsection{The ergodic BSDE with zero Neumann boundary conditions in a weakly dissipative environment}
In a first time we are concerned with the following EBSDE with zero Neumann condition in infinite horizon:
\begin{equation}\label{EBSDE with zero neumann condition}
Y_t^{x} = Y_T^{x} + \int_t^T[ \psi(X_s^{x},Z_s^{x})- \lambda] \der s - \int_t^T Z_s^x \der W_s, ~ \forall 0 \leq t \leq T < +\infty,
\end{equation}
where the unknown is the triplet $(Y_\cdot^x,Z_\cdot^x,\lambda)$. $(X_t^x)_{t \geq 0}$  is the solution of (\ref{SDE reflected}).

To study the problem of existence of a solution to such an equation, we are going to study the following BSDE, with monotonic drift in $y$, regularized coefficients and penalized generator, namely: $\forall 0 \leq t \leq T < +\infty$,
\begin{equation}\label{BSDE alpha approchant l'EBSDE avec générateur réfléchi}
Y_t^{x,\alpha,n,\varepsilon} = Y_T^{x,\alpha,n,\varepsilon} + \int_t^T[ \psi^{\varepsilon}(X_s^{x,n,\varepsilon},Z_s^{x,\alpha,n,\varepsilon})- \alpha Y_s^{x,\alpha,n,\varepsilon}] \der s - \int_t^T Z_s^{x,\alpha,n,\varepsilon} \der W_s,
\end{equation}
where the process $(X_t^{x,n,\varepsilon})$ is the solution of the following SDE:
\[X_t^{x,n,\varepsilon} = x + \int_0^t (f^{\varepsilon}(X_s^{x,n,\varepsilon}) + F_n^{\varepsilon}(X_s^{x,n,\varepsilon}))\der s + \int_0^t \sigma(X_s^{x,n}) \der W_s. \]

\begin{rema}
$F_n$ is regularized like other regularized functions. Thanks to convolutions arguments it is possible to construct a sequence of functions $F_n^{\varepsilon}$ which converges pointwisely toward $F_n$ and such that for all $\varepsilon$, $F_n^{\varepsilon}$ is $0$-dissipative and  $4n$-Lipschitz.
\end{rema}

Now we can state the existence theorem for EBSDE (\ref{EBSDE with zero neumann condition}).

\begin{théo}\label{thm existence avec mu nul}
Assume that the hypotheses \ref{hypo f weakly dissipative, sigma invertible}, \ref{hypo psi lipschitz en x,z et borné en x}, \ref{hypo G convexe} and \ref{hypo bord de G} hold. Then there exists a solution $(\overline{Y}_t^{x},\overline{Z}_t^{x},\overline{\lambda})$ to EBSDE (\ref{EBSDE with zero neumann condition}) such that $\overline{Y}_\cdot^x=\overline{v}(V_\cdot^x)$ with $\overline{v}$ locally Lipschitz, and there exists a measurable function $\overline{\xi} : \R^{d} \rightarrow \R^{1 \times d}$ such that $\overline{Z}^x \in \M$ and $\overline{Z}_\cdot^x = \overline{\xi}(X_\cdot^x)$.
\end{théo}

\begin{proof}
We give the main ideas, because the proof is very similar to the proof of Theorem \ref{théorème existence EDSRE cas général avec f dissipatif}. The beginning of the proof is the same as the proof of Theorem \ref{théorème existence EDSRE cas général avec f dissipatif}. Lemma \ref{lemm existence unicité BSDE monotone alpha} gives us the existence and uniqueness of the solution $(Y^{x,\alpha,n,\varepsilon},Z^{x,\alpha,n,\varepsilon})$ of BSDE (\ref{BSDE alpha approchant l'EBSDE avec générateur réfléchi}) in $\SSS \times \M$. Then, as the function $d + F_n$ is still $\eta$-dissipative and as the work in the previous section involves $d$ only through its dissipativity constant $\eta$, we can apply previous results. As always we define $v^{\alpha,n,\varepsilon}(x) := Y_0^{\alpha, n , \varepsilon}$. By Lemma \ref{Lemm estimates on v^alpha} we have the following estimate: $\forall x, x' \in \R^d$:
\begin{align*}
|v^{\alpha,n,\varepsilon}(x) - v^{\alpha,n,\varepsilon}(x')| \leq C(1+|x|^2 + |x'|^2)|x-x'|.
\end{align*}
In addition we also have
\begin{align*}
|\alpha v^{\alpha,n,\varepsilon}(0)| \leq M_\psi.
\end{align*}

As those inequalities are uniform in $\varepsilon$ it is possible to construct by a diagonal procedure a subsequence $\varepsilon_p \rightarrow + 0$ such that  $\forall n \in \N  ,\alpha > 0$:
\[ \overline{v}^{\alpha,n,\varepsilon_p}(x) \underset{p \rightarrow + \infty}{\longrightarrow} \overline{v}^{\alpha,n}(x)    \t{~~~and~~~} \alpha v^{\alpha,n,\varepsilon_p(0)} \underset{p \rightarrow + \infty}{\longrightarrow} \overline{\lambda}^{\alpha,n},\]
We recall the fact that the function $\overline{v}^{\alpha,n}$ is locally Lipschitz on $\R^d$ and that we keep the following estimates:
\[|\overline{v}^{\alpha,n}(x)-\overline{v}^{\alpha,n}(x')| \leq C(1+|x|^2+|x'|^2)|x-x'|;\]
\[|\overline{\lambda}^{\alpha,n}| \leq M_\psi.\]

Now let us define $\forall t \geq 0$, $\overline{Y}_t^{x,\alpha,n} := \overline{v}^{\alpha,n}(V_t^{x,n})$. Let us show that 
 \[\E\int_0^T|\overline{Y}_s^{x,\alpha,n,\varepsilon_p} - \overline{Y}_s^{x,\alpha,n}|^2\der s \underset{p \rightarrow + \infty}{\rightarrow} 0 \t{~~and~~}  \E|\overline{Y}_T^{x,\alpha,n,\varepsilon_p} - \overline{Y}_T^{x,\alpha,n}|^2 \underset{p \rightarrow + \infty}{\rightarrow} 0.\]
First we write: 
\begin{align*}
|\overline{v}^{\alpha,n,\varepsilon_p}(X_s^{x,n,\varepsilon_p}) - \overline{v}^{\alpha,n}(X_s^{x,n})| &\leq |\overline{v}^{\alpha,n,\varepsilon_p}(X_s^{x,n,\varepsilon_p}) - \overline{v}^{\alpha,n,\varepsilon_p}(X_s^{x,n})|\\
& ~~~~~~~~~~ + |\overline{v}^{\alpha,n,\varepsilon_p}(X_s^{x,n})-\overline{v}^{\alpha,n}(X_s^{x,n})|\\
& \leq C(1+|X_s^{x,n,\varepsilon_p}|^2+|X_s^{x,n}|^2)|X_s^{x,n,\varepsilon_p}-X_s^{x,n}|,
\end{align*}
which shows the pointwise convergence of $\overline{v}^{\alpha,n,\varepsilon_p}(V_s^{x,n,\varepsilon_p})$ toward $\overline{v}^{\alpha,n}(V_s^{x,n})$ almost surely when $p \rightarrow + \infty$.
Then, due to the fact that $|\overline{v}^{\alpha,\beta(\varepsilon_n)}(V_s^{x,\beta(\varepsilon_n)})| \leq M_\psi/\alpha$ $\Pb$-a.s., we can apply the dominated convergence theorem to show that:
\[\E\int_0^T|\overline{Y}_s^{x,\alpha,n,\varepsilon_p} - \overline{Y}_s^{x,\alpha,n}|^2\der s \underset{p \rightarrow + \infty}{\rightarrow} 0 \t{~~and~~}  \E|\overline{Y}_T^{x,\alpha,n,\varepsilon_p} - \overline{Y}_T^{x,\alpha,n}|^2 \underset{p \rightarrow + \infty}{\rightarrow} 0.\]
In addition it is possible to show as in Theorem \ref{théorème existence EDSRE cas général avec f dissipatif} that $(Z^{x,\alpha,n,\varepsilon_p})_p$ is Cauchy in $\M$.

Note that we keep the estimates $\forall x, x' \in \R^d$:
\begin{align*}
|\overline{v}^{\alpha,n}(x) - \overline{v}^{\alpha,n}(x')| \leq C(1+|x|^2 + |x'|^2)|x-x'|,
\end{align*}
and
\begin{align*}
|\overline{\lambda}^{\alpha,n}| \leq M_\psi.
\end{align*}

Therefore, again, by a diagonal procedure, it is possible to extract a subsequence $(\beta(n))_n$ such that 
\begin{align*}
v^{\alpha,\beta(n)}(x) \rightarrow \overline{v}^{\alpha}(x).\\
\end{align*}
And thanks to Lemma \ref{lemm convergence processus prenalise}, one can apply the dominated convergence theorem to show that:
\[\E\int_0^T|\overline{Y}_s^{x,\alpha,\beta(n)} - \overline{Y}_s^{x,\alpha}|^2\der s \underset{n \rightarrow + \infty}{\rightarrow} 0 \t{~~and~~}  \E|\overline{Y}_T^{x,\alpha,\beta(n)} - \overline{Y}_T^{x,\alpha}|^2 \underset{n \rightarrow + \infty}{\rightarrow} 0.\]

Finally a last diagonal procedure in $\alpha$ allow us to conclude (see the end of the proof of Theorem \ref{théorème existence EDSRE cas général avec f dissipatif}).

\end{proof}


Once again, we notice that the solution we have constructed satisfies the following growth property:
\begin{align*}
|\overline{Y}_t^x| \leq C(1+|X_t^x|^2),
\end{align*}
so it is natural to establish the following theorem under the same growth properties.

\begin{théo}\label{théorème unicité lambda condition de neuman nulle}
(Uniqueness of $\lambda$). Assume that the hypotheses \ref{hypo f weakly dissipative, sigma invertible} and \ref{hypo psi lipschitz en x,z et borné en x} hold true. Let $(Y,Z,\lambda)$ be a solution of EBSDE (\ref{EBSDE with zero neumann condition}). Then $\lambda$ is unique among solutions $(Y,Z,\lambda)$ such that $Y$ is a bounded continuous process and $Z \in \M$. Finally assume that we have the following growth property 
\begin{align*}
&|Y_t| \leq C(1+|X_t^x|^2),\\
&|Y'_t| \leq C'(1+|X_t^x|^2).
\end{align*}
Then $\lambda = \lambda'$.
\end{théo}
\begin{proof}
Simply, adapt the proof of Theorem 4.6 of \cite{FUHRMAN_HU_TESSITORE_ERGODIC_BSDE}. With the same notations we can write: 
\begin{align*}
\widetilde{\lambda} &= T^{-1} \E^{\Pb_h}[\widetilde{Y}_T - \widetilde{Y}_0] \\
& \leq (C+C')T^{-1}(2+|x|^2 + \E^{\Pb_h}|X_T^x|^2)\\
& \leq (C+C')T^{-1}(2+|x|^2 + \E^{\Pb_h}|X_T^{x,n}|^2 + \E^{\Pb_h}|X_T^x - X_T^{x,n}|^2 )
\end{align*}
To conclude, just use the first estimate from Lemma \ref{lemm estimates dissipativity}, the estimate from Lemma \ref{lemm convergence processus prenalise} and let $T \rightarrow + \infty$.
\end{proof}

\subsection{The ergodic BSDE with non-zero Neumann boundary conditions in a weakly dissipative environment}
We are now concerned by the following EBSDE in infinite horizon:
\begin{equation}\label{EBSDE with non-zero neumann condition}
Y_t^{x} = Y_T^{x} + \int_t^T[ \psi(X_s^{x},Z_s^{x})- \lambda] \der s + \int_t^T[g(X_s^x)-\mu] \der K_s^x - \int_t^T Z_s^x \der W_s, ~ \forall 0 \leq t \leq T < +\infty,
\end{equation}
where $g: \R^d \rightarrow \R ~~\t{is measurable}$ and such that the term 
$\int_t^T[g(X_s^x)-\mu] \der K_s^x$ is well defined for all $0\leq t \leq T < +\infty $.

\begin{prop}\label{théorème existence sol (Y,Z,lambda)}
(Existence of a Solution $(Y,Z,\lambda)$). Assume that  the hypothesis \ref{hypo f weakly dissipative, sigma invertible}, \ref{hypo psi lipschitz en x,z et borné en x} and \ref{hypo bord de G} hold true. Then for any $\mu \in \R$ there exists $\lambda \in \R$, $Y^x$ continuous adapted process  and $Z^x \in \M$  such that the triple $(Y,Z,\lambda)$ is a solution of EBSDE (\ref{EBSDE with non-zero neumann condition}).
\end{prop}
\begin{proof}
The Theorem \ref{thm existence avec mu nul} gives us the existence of a solution $(Y^x,Z^x,\lambda)$ of the following EBSDE 
 \begin{equation}
Y_t^{x} = Y_T^{x} + \int_t^T[ \psi(X_s^{x},Z_s^{x})- \lambda] \der s - \int_t^T Z_s^x \der W_s, ~~~~ \forall 0 \leq t \leq T < +\infty.
\end{equation}
Now, defining $\widehat{Y}_t^{x} = Y_t^x - \int_0^t [g(X_s^x) - \mu] \der K_s^x $, it is easy to see that $(\widehat{Y}^x,Z^x,\lambda)$ is a solution of the EBSDE (\ref{EBSDE with non-zero neumann condition}) with $\mu$ fixed. 
\end{proof}
\begin{rema}
The constructed solution $\widehat{Y}^{x}$ is not Markovian anymore. Furthermore, it satisfies the following growth property: $\forall t \geq 0$,
$|\widehat{Y}_t^x| \leq C(1+ |X_t^x|^2 +K_t^x)$. This dependence on $K_t^x$ prevents us to get the uniqueness of $\lambda$ among the space of solutions satisfying such a growth property.
\end{rema}

Similarly, for every $\lambda \in \R$, an existence result can be stated for a solution $(Y,Z,\mu)$.
\begin{prop}\label{théorème existence sol (Y,Z,mu)}
(Existence of a Solution $(Y,Z,\mu)$). Assume that the hypotheses \ref{hypo f weakly dissipative, sigma invertible}, \ref{hypo psi lipschitz en x,z et borné en x} and \ref{hypo bord de G} hold true. Then for any $\lambda \in \R$ there exists a continuous adapted process $Y$  and $Z^x \in \M$  such that for all $\mu \in \R$ the triple $(Y,Z,\mu)$ is a solution of EBSDE (\ref{EBSDE with non-zero neumann condition}).
\end{prop}
\begin{proof}
From Theorem \ref{thm existence avec mu nul}, we have constructed  a solution $(Y^{x,0},Z^{x,0},\lambda^0)$ of the following EBSDE 
  \begin{equation}
Y_t^{x,0} = Y_T^{x,0} + \int_t^T[ \psi(X_s^{x},Z_s^{x,0})- \lambda^0] \der s - \int_t^T Z_s^{x,0} \der W_s, ~~~~ \forall 0 \leq t \leq T < +\infty.
\end{equation}
Then setting $\widehat{Y}_t^x := Y_t^{x,0} + (\lambda-\lambda^0)t - \int_0^t [g(X_s^x)-\mu]\der K_s^x$, the triple $(\widehat{Y}^x,Z^{x,0},\mu)$ is solution of  EBSDE (\ref{EBSDE with non-zero neumann condition}).
\end{proof}

\begin{rema}
The constructed solution satisfies the following growth property:
\begin{align*}
|\widehat{Y}_t| \leq C(1+|X_t^x|^2+K_t^x+t), \mathbb{P}\text{-a.s}.
\end{align*}
Again, this solution does not allow us to establish a result of uniqueness for $\mu$ among the space of solutions satisfying such a growth property.
\end{rema}

\begin{rema}
If the convex $\G$ is assumed to be bounded, it is possible, following \cite{Richou_Ergodic_BSDE_PDe_Neumann} to show that there exists a Markovian  solution $(Y,Z,\lambda)$ when $\mu$ is fixed or $(Y,Z,\mu)$ when $\lambda$ is fixed exists, for a driver weakly dissipative. The proofs are the same as in \cite{Richou_Ergodic_BSDE_PDe_Neumann}.
\end{rema}


\subsection{Probabilistic interpretation of the solution of an elliptic PDE with zero Neumann boundary condition}
We are concerned with the following semi-linear elliptic PDE:
\begin{align}\label{PDE avec condition de Neumann au bord}
\left\{ \begin{array}{ll}
\LL v(x) + \psi(x, \nabla v(x)\sigma(x)) = \lambda, &x \in G, \\
\frac{\partial v}{\partial n}(x)  = 0, & x \in \partial G,
\end{array}\right.
\end{align}
where:
\begin{align*}
\LL u(x) = \frac{1}{2}\tr (\sigma(x) \transpose \sigma(x) \nabla^2u(x)) +   \nabla u(x) f(x).
\end{align*}
The unknowns of this equation is the couple $(v,\lambda)$. Now we show that the pair $(v,\lambda)$ defined in Theorem \ref{théorème existence sol (Y,Z,lambda)} is a viscosity solution of the PDE (\ref{PDE avec condition de Neumann au bord}).

\begin{théo}
Assume that hypotheses of Theorem \ref{thm existence avec mu nul} hold 
. Then $(v,\lambda)$ is a viscosity solution of the elliptic PDE (\ref{PDE avec condition de Neumann au bord}) where $(v,\lambda)$ is defined in Theorem \ref{thm existence avec mu nul}.
\end{théo}
\begin{proof}
Just adapt the proof of Theorem $4.3$ from \cite{PARDOUXPENGADAPTED}.
\end{proof}


\subsection{Optimal ergodic control}
We make the standard assumption for optimal ergodic control, namely we consider $U$ a separable metric space, which is the state space of the control process $\rho$. $\rho$ is assumed to be $(\mathscr{F}_t)$-progressively measurable. We introduce $R : U \rightarrow \R^d$ and $L : \R^d \times U \rightarrow \R$ two continuous functions such that , for some constants $M_R > 0$ and $M_L > 0$, $\forall u \in U, \forall x,x' \in \R^d$,
\begin{itemize}
\item $|R(u)| \leq M_R$,
\item $|L(x,u)| \leq M_L$,
\item $|L(x,u) - L(x',u)| \leq M_L|x-x'|$.
\end{itemize}

For an arbitrary control $\rho$, the cost will be evaluated relatively to the following Girsanov density:
\begin{align*}
\Gamma_T^{\rho} = \exp \left( \int_0^T R(\rho_s)\der W_s - \frac{1}{2}\int_0^T |R(\rho_s)|^2\der s \right).
\end{align*}
We denote by $\Pb_T^{\rho}$ the associated probability measure, namely: $\der \Pb_T^{\rho} = \Gamma_T^{\rho} \der \Pb$ on $\mathscr{F}_T$. Now we define the ergodic costs, relatively to a given control $\rho$ and a starting point $x \in \R^d$, by:
\begin{align}
I(x,\rho) = \limsup_{T \rightarrow + \infty} \frac{1}{T}\E_T^{\rho} \left[ \int_0^T L(X_s^x,\rho_s) \der s \right],
\end{align}

where $\E_T^{\rho}$ denotes expectation with respect to $\Pb_T^{\rho}$. We notice that the process $W_t^{\rho} := W_t - \int_0^t R(\rho_s) \der s$ is a Wiener process on $[0,T]$ under $\Pb_T^{\rho}$.
We define the Hamiltonian in the usual way:
\begin{align}\label{definition hamiltonien}
\psi(x,z) = \inf_{u \in U} \{ L(x,u) + zR(u) \}, ~~~~x \in R^{d}, ~~z\in \R^{1 \times d},
\end{align}
and we remark that if, for all $x$, $z$, the infimum is attained in (\ref{definition hamiltonien}) then, according to Theorem $4$ of \cite{FILIPPOV_IMPLICIT_FUNCTION_LEMMA}, there exists a measurable function $\gamma : \R^d \times \R^{1 \times d} \rightarrow U$ such that:
\begin{align}\label{hamilonian}
\psi(x,z) = L(x,\gamma(x,z)) + zR(\gamma(x,z)).
\end{align}
One can verify that $\gamma$ is a Lipchitz function. Now we can prove the following theorem, exactly like in \cite{Richou_Ergodic_BSDE_PDe_Neumann}.

\begin{théo}
Assume that the hypotheses of Theorem \ref{thm existence avec mu nul} hold true. Let $(Y,Z,\lambda)$ be a solution of EBSDE (\ref{EBSDE with non-zero neumann condition}) with $\mu$ fixed. Then:
\begin{enumerate}\vspace{-0.3 cm}
\item For arbitrary control $\rho$ we have $I(x,\rho) \geq \lambda$.
\item If $L(X_t^x,\rho_s) + Z_t^x R(\rho_t) = \psi (X_t^x,Z_t^x)$, $\Pb$-a.s. for almost every $t$ then  $I(x,\rho) = \lambda$.
\item If the infimum is attained in (\ref{hamilonian}) then the control $\overline{\rho}_t = \gamma (X_t^x,Z_t^x)$ verifies $I(x,\overline{\rho}) = \lambda$.
\end{enumerate}
\end{théo}

\begin{rema}
When the Neumann conditions are different from $0$, we need regularity on the solution $Y_t^x$ in order to state the same result. Again the degeneracy of the solution constructed in Proposition \ref{théorème existence sol (Y,Z,lambda)} or \ref{théorème existence sol (Y,Z,mu)} does not allow us to conclude.
\end{rema}


\section{Appendix}
\appendix
\section{Proof of Lemma \ref{lemm estimates dissipativity}}
Let us define $\varphi (x) = |x-a|^p$ for $p \geq 1$. We recall the following formulas for derivatives of $\varphi$, for $p \geq 2$. 
\[\nabla \varphi (x) = p(x-a)|x-a|^{p-2} .\]
\[
{\partial^2 \varphi (x)}/{\partial x_i \partial x_j} = 
\left \{
\begin{array}{ ll}
p|x-a|^{p-2} + p(p-2)(x_i - a_i)^2|x-a|^{p-4}  &~~\t{if}~~i=j,\\
p(p-2)(x_i-a_i)(x_j-a_j)|x-a|^{p-4} &~~\t{if}~~i \neq j.
\end{array}\right.
\]
 Therefore we have the following estimate 
 \begin{align}\label{estimée sur l hessienne}
  |\nabla^2 \varphi (x)| \leq K|x-a|^{p-2},
 \end{align}
for a constant $K$ which depends only on $p$ and $d$.
Under the hypothesis of this Lemma, it is well known that a unique strong solution for which the explosion time is almost surely equal to infinity exists (see \cite{MAO_SZPRUCH_SRONG_CONVERGENCE_NON_GLOBALLY} for example).
By Itô's formula we get, for $p =2$, for all $t \geq 0$,
\begin{align}\label{ito Vs - a au carre}
|V_t^x-a|^2e^{2 \eta_1 t} &= |x-a|^2 + 2\int_0^t e^{2 \eta_1 s}(V_s^x-a , f(V_s^x) \der s + \sigma(V_s^x) \der W_s)\nonumber\\
&~~~~~~ + 2\eta_1\int_0^t |V_s^x-a|^2 e^{2 \eta_1 s} \der s + \int_0^t \sum_{i}(\sigma(V_s^x) \transpose\sigma(V_s^x))_{i,i}e^{2 \eta_1 s}\mathrm{d}s\nonumber\\ 
&\leq |x-a|^2  + 2\int_0^t \transpose(V_s^x-a)\sigma(V_s^{x}) \der W_s + \frac{2\eta_2+d|\sigma|_\infty}{2 \eta_1} (e^{2 \eta_1 t}-1). 
\end{align}
Taking the expectation, we get:
\begin{align}\label{equation non factorise}
\E|V_t^x - a|^2 \leq |x-a|^2e^{-2\eta_1 t} + \frac{2\eta_2+ d|\sigma|_\infty^2 }{2\eta_1}(1-e^{-2\eta_1 t}).
\end{align}
Therefore:
\begin{equation}\label{estimée uniforme en temps pour p=2}
\E|V_t^x|^2 \leq  C(1+|x|^2e^{-2\eta_1 t}),
\end{equation}
where $C$ is a constant that depends only on $a$, $\eta_1$, $\eta_2$ and $\sigma$ but not on the time $t$.

Let $0 < \delta < p \eta_1$. For $p > 2$, Itô's formula gives us, for a generic constant $C$ which depends only on $p$, $d$, $|\sigma|_{\infty}$, $\eta_2$, $\varepsilon$ (defined later):
\begin{align*}
|V_t^x-a|^pe^{(p \eta_1 - \delta)t} &\leq  |x-a|^p + p\int_0^t e^{(p \eta_1 - \delta)s} |V_s^x-a|^{p-2}(V_s^x-a , f(V_s^x)\der s+\sigma(V_s^x) \der W_s)\\
&~~~~~~ + (p \eta_1 - \delta)\int_0^t |V_s^x-a|^p e^{(p\eta_1-\delta)} \der s\\
&~~~~~~+ \frac{1}{2}\int_0^t \tr (\sigma(V_s^x) \transpose \sigma(V_s^x)\nabla^2 \varphi(V_s^x))e^{(p\eta_1 - \delta)s} \der s.
\end{align*}
Then, taking the expectation, using the assumption on $f$ and using estimate (\ref{estimée sur l hessienne}) we have
\begin{align*}
\E|V_t^x-a|^pe^{(p \eta_1 - \delta)t} \leq  |x-a|^p &+ C\int_0^t \E|V_s^x-a|^{p-2} e^{(p\eta_1 - \delta)s}\der s \\
&- \delta \int_0^t \E|V_s^x-a|^{p} e^{(p \eta_1 - \delta)s} \der s.
\end{align*}
The Young inequality  $ab \leq  a^p/p + b^q/q$ for $1/p + 1/q = 1$ with $p$ replaced by $p/(p-2)$ and $q$ replaced by $p/2$ applied to the last term of the above inequality allows us to write:
\[
|V_s^{x}-a|^{p-2} \leq (p-2)\varepsilon|V_s^{x}-a|^p/p + 2/(p \varepsilon^{(p-2)/2}),
\]
hence,
\begin{align*}
\E|V_t^x-a|^pe^{(p \eta_1 - \delta)t} &\leq  |x-a|^p + \varepsilon C \int_0^t \E|V_s^x-a|^p e^{(p\eta_1 - \delta)s} \der s + C/\varepsilon^{(p-2)/2} \\
&~~~~~~~~~~~~~~~~~~~~~~~~~~~~~~~~~~~~~~ - \delta\int_0^t  \E|V_s^x-a|^{p} e^{(p \eta_1 - \delta)s} \der s.
\end{align*}
We choose  $\varepsilon = \delta/C$, then:
\begin{align*}
\E|V_t^x-a|^pe^{(p \eta_1 - \delta)t} \leq  |x-a|^p C.
\end{align*}
Therefore:
\[
\E|V_t^x-a|^p \leq  C(1+|x|^pe^{-(p\eta_1 - \delta)t}).
\]
This can be rewritten:
\[
\E|V_t^x|^p \leq  C(1+|x|^pe^{-(p\eta_1 - \delta)t}),
\]
where $C$ is a constant which depends on $p$, $d$, $\sigma_{\infty}$, $\eta_1$, $\eta_2$, $\varepsilon$ and $a$. Finally, note that this result holds for any $0 <\delta <p\eta_1$.

Now, let us come back to (\ref{ito Vs - a au carre}), we have, for all $r >1$, $2r = p$
\begin{align*}
\E [\sup_{0\leq t \leq T}|V_t^x|^p ] &\leq C\left(1+|x|^p + \E \sup_{0 \leq t \leq T}\left| \int_0^t \transpose (V_s^x-a)\sigma(V_s^{x}) \der W_s\right|^p\right)\\
&\leq C\left[1+|x|^p + \E \left(\left( \int_0^T |\transpose (V_s^x-a)\sigma(V_s^{x})|^2 \der s\right)^{r/2}\right)\right]
\end{align*}
by BDG's inequality. Now distinguish the cases $r/2 < 1$ or $r/2 \geq 1$, we readily get, for each $p>2$, 
\begin{align*}
\E [\sup_{0\leq t \leq T}|V_t^x|^p ] &\leq C(T)(1+|x|^p+|x|^r) \\
& \leq C(T)(1+|x|^p).
\end{align*}
Once this is established, one can readily extend this estimate to the case $p \geq 1$. Indeed, for $0< \alpha < 1$, we have, by Jensen's inequality
\begin{align*}
\E [\sup_{0\leq t \leq T}|V_t^x|^{p \alpha} ]  &\leq   \left(\E \left[\sup_{0\leq t \leq T}|V_t^x|^{p} \right]\right)^{\alpha} \\
& \leq C(T)(1+|x|^{p \alpha}).
\end{align*}
\qed

\section{Proof of Lemma \ref{lemm estimates on semi groupe kolmogorov}}

We adapt the proof of Theorem $2.4$ from \cite{Debu_Hu_Tess_weak_dissipative}. We give the full proof for reader convenience. In this proof, $\kappa_i$, $i=0,1...$ denotes a constant which depends only on $\eta$, $B$, $\sigma_\infty$ and the dimension $d$. There are three steps in this proof. In the first step, we show that the process $V^x$ enters a fixed ball quickly enough. In the second step, we construct a coupling of solutions starting from two different points in this ball and we show that the probability of the constructed solutions to be equal after a time $T$ (given in the proof) is positive. Iterating this argument in step $3$, we obtain the result.

\textbf{Step 1 :}
By Remark \ref{rema dissipativity implies previous conditions}, one can take $a=0$ in equation (\ref{equation non factorise}) and then:
\begin{align*}
\E|V_t^x|^2 \leq |x|^2e^{-\eta_1 t} + \kappa_1.
\end{align*}
By the Markov property : $ \forall k \in \N$,
\begin{align}\label{B1}
\E[|V_{(k+1) T}^x|^2 | \F_{k T}] \leq |V_{k T}^x|^2e^{-\eta_1 T} + \kappa_1.
\end{align}
Let us define for $R \geq 0$,
\begin{align*}
C_k = \{ |V_{k T}^x|^2 \geq R \}, ~~~~ B_k = \bigcap_{j=0}^k C_j.
\end{align*}
By Markov's inequality
\begin{align}\label{B2}
\Pb(C_k+1 | \F_{kT}) &\leq \frac{\E(|V^x_{(k+1)T} |^2 | \F_{kT})}{R} \nonumber \\
&\leq \frac{|V_{k T}^x|^2e^{-\eta_1 T}}{R}  + \frac{\kappa_1}{R}. 
\end{align}
Let us multiply (\ref{B1}) and (\ref{B2}) by $\mathds{1}_{B_k}$ and let us take the expectation to obtain, since $\mathds{1}_{B_{k+1}} \leq \mathds{1}_{B_k}$
\begin{align*}
\begin{pmatrix}
 \E( |V_{(k+1)T}^x|^2 \mathds{1}_{B_{k+1}} ) \\
 \Pb(B_{k+1}) 
\end{pmatrix}
\leq 
 A \begin{pmatrix}
 \E( |V_{kT}^x|^2 \mathds{1}_{B_{k}} ) \\
 \Pb(B_{k}) 
\end{pmatrix}
\end{align*}
where 
\begin{align*}
A = \begin{pmatrix}
e^{-\eta_1 T} & \kappa_1\\
\frac{e^{-\eta_1 T}}{R} & \frac{\kappa_1}{R}
\end{pmatrix}.
\end{align*}
The eigenvalues of $A$ are $0$ and $ e^{- \eta_1 T} + \frac{\kappa_1}{R}$. So for every $T >0$ we can pick $R$ large enough such that $\lambda_{T,R} := e^{- \eta_1 T} + \frac{\kappa_1}{R} <1$. 

We deduce that 
\begin{align*}
\Pb(B_k) \leq \kappa_2 \left( \lambda_{T,R} \right)^k(1+|x|^2).
\end{align*}
Defining,  
\begin{align*}
\tau = \inf \{kT ; |V_{kT}^x|^2 \leq R \},
\end{align*}
it follows that 
\begin{align*}
\Pb(\tau \geq kT) \leq \Pb(B_k) \leq \kappa_2 (\lambda_{T,R})^k (1+|x|^2). 
\end{align*}
Thus, if we pick $\alpha > 0$ small enough such that $e^{\alpha T} \lambda_{T,R} < 1$, we obtain:
\begin{align}\label{esperance exponentiel}
\E(e^{\alpha \tau}) \leq \kappa_2 (1+|x|^2) \sum_{k=0}^{+\infty} \left( e^{\alpha T} \lambda_{T,R} \right)^k \leq \kappa_3(1+|x|^2).
\end{align}
\textbf{Step 2.} In this step, we construct a coupling of processes starting respectively from $x,y \in B_R$, the ball of center $0$ and radius $R$, such that  the probability that they take the same value at time $T$ is strictly positive. We denote by $\mu^x$ the law of $V^x$ under $\Pb$ and by $\mu^y$ the law of $V^y$ under $\Pb$ on $[0,T]$. Let us define 
\begin{align*}
\widetilde{V}_t = V_t^{y} + Y_t^{x,y}
\end{align*}
where $Y_t^{x,y}$ is the solution of the following SDE, for all $t < T$,
\begin{equation}\label{eq Y}
  \left\{
      \begin{aligned}
       \der Y_t^{x,y} &= \left[d(V_t^y + Y_t^{x,y}) - d(V_t^y)-\sigma(V_t^y+Y_t^{x,y})\sigma^{-1}(V_t^y) \frac{L Y_t^{x,y}}{T-t}\right]\der t \\
       &~~~~+ \left[\sigma(V_t^y + Y_t^{x,y})- \sigma(V_t^y)\right]\der W_t,\\
        Y_0^{x,y} &= x-y,\\
      \end{aligned}
    \right.
\end{equation}
and where $L > 0$. We denote by $\widetilde{\mu}$ the law of $\widetilde{V}$ on $[0,T)$.
The process $\widetilde{V}_t$ satisfies for all $t < T$:
\begin{equation}\label{eq V tilde}
  \left\{
      \begin{aligned}
\der \widetilde{V}_t &= f(\widetilde{V}_t) \der t + \sigma(\widetilde{V}_t)\left[\der W_t + \left(\sigma^{-1}(\widetilde{V}_t)\left(b(V_t^y) - b(\widetilde{V}_t)\right) - \sigma^{-1}(V_t^y) \frac{LY_t^{x,y}}{T-t}\right)\der t\right].  \\
       &~~~~+ \left[\sigma(V_t^y + Y_t^{x,y})- \sigma(V_t^y)\right]\der W_t,\\
        \widetilde{V}_0 &= x,\\
      \end{aligned}
    \right.
\end{equation}
Since all the coefficients are locally Lipschitz, $Y_t^{x,y}$ (and thus $\widetilde{V}_t$) are well-defined continuous process for $t \leq T \wedge \zeta$ where $\zeta$ is the explosion time of $Y_t^{x,y}$; namely, $\zeta := \lim_{n \rightarrow + \infty} \zeta_n$ for 
\begin{align*}
\zeta_n := \inf \{ t \in [0,T) : |Y_t^{x,y}| \geq n \}
\end{align*}
where we set $\inf \{ \varnothing \} = T$. We define
\begin{align*}
h(t) = \sigma^{-1}(\widetilde{V}_t)\left(b(V_t^y) - b(\widetilde{V}_t)\right) - \sigma^{-1}(V_t^y) \frac{LY_t^{x,y}}{T-t}, ~~~~ t \leq T \wedge \zeta,
\end{align*}
\begin{align*}
\der \widetilde{W}_t = \der W_t + h(t) \der t,~~~~ t \leq T \wedge \zeta,
\end{align*}
and 
\begin{align*}
I_t= \int_0^t h(s) \der W_s, ~~~~ t \leq T \wedge \zeta.
\end{align*}
If $\zeta = T$ and
\begin{align*}
R_t := \exp\left( -I_t - \frac{1}{2} \langle I,I \rangle_t \right)
\end{align*}
is a uniformly integrable martingale for $t \in [0,T)$, then by the martingale convergence theorem, $R_T := \lim_{t \nearrow T} R_t$ exists and $(R_t)_{t \in [0,T]}$ is a martingale. In this case, by the Girsanov theorem $(\widetilde{W}_t)_{t \in [0,T]}$ is a standard Brownian motion under the probability $R_T \Pb$. 
Rewrite (\ref{eq Y}) as 
\begin{equation}\label{equation Y sous Q}
  \left\{
      \begin{aligned}
       \der Y_t^{x,y} &= \left[d(\widetilde{V}_t) - d(V_t^y)- \left(\sigma(\widetilde{V}_t)- \sigma(V_t^y)\right)(b(\widetilde{V}_t) - b(V_t^y))  \right]\der t - \frac{L Y_t^{x,y}}{T-t}\der t \\
       &~~~~ + \left[\sigma(\widetilde{V}_t)- \sigma(V_t^y)\right]\der \widetilde{W}_t,\\
        Y_0^{x,y} &= x-y,\\
      \end{aligned}
    \right.
\end{equation}

Now we would like to apply the Girsanov theorem. The following lemma is a direct adaptation of a result in \cite{WANG_HARNACK}. However we give the proof for completeness.
\begin{lemm}
Assume Hypothesis \ref{hypo f weakly dissipative, sigma invertible} hold true and let $x,y \in \R^d$ and $T>0$ be fixed. Then
\begin{enumerate}
\item There holds
\begin{align*}
\sup_{t \in [0,T), n \geq 1} \E R_{t \wedge \zeta_n} \log R_{t \wedge \zeta_n} < + \infty.
\end{align*}
Consequently,
\begin{align*}
R_{t \wedge \zeta} := \lim_{n \nearrow \infty} R_{t \wedge \zeta_n \wedge (T- 1/n)}, ~~~~ t \in [0,T], ~~~~~R_{T \wedge \zeta} := \lim_{s \nearrow T} R_{s \wedge \zeta}
\end{align*}
exist such that $(R_{s \wedge \zeta })_{s \in [0,T]}$ is a uniformly martingale.
\item Let $\Q = R_{T \wedge \zeta} \Pb$. Then $\Q(\zeta = T) = 1$ so that $\Q = R_T \Pb$.
\item $Y_T^{x,y} = 0, ~~~~ \Q \text{-a.s.}$
\end{enumerate}
\end{lemm}
\begin{proof}
($1$) let $t \in [0,T)$ and be fixed. By an Itô's formula,
\begin{align*}
\frac{|Y_{t \wedge \zeta_n}^{x,y}|^2}{T-s} &= \frac{|x-y|^2}{T} +  2\int_0^{t \wedge \zeta_n} \left(\frac{Y_s^{x,y}}{T-s} , d(\widetilde{V}_t) - d(V_s^y) - \frac{L Y_s^{x,y}}{T-s}\right)\der s \\
&~~~~ - 2\int_0^{t \wedge \zeta_n} \left(\frac{Y_s^{x,y}}{T-s}, \left(\sigma(\widetilde{V}_t)-\sigma(V_s^y)\right)(b(\widetilde{V}_t) - b(V_t^y)) \right) \der s \\
&~~~~ + \int_0^{t \wedge \zeta_n} \frac{|Y_s^{x,y}|^2}{|T-s|^2} \der s + 2\int_0^{t \wedge \zeta_n}  \left(\frac{Y_s^{x,y}}{T-s},\left[\sigma(\widetilde{V}_s)- \sigma(V_s^y)\right]\der \widetilde{W}_s \right) \\
&~~~~ + \int_0^{t \wedge \zeta_n} \frac{1}{T-s}\tr\left[(\sigma(\widetilde{V}_s) - \sigma(V_s^y))\transpose(\sigma(\widetilde{V}_s)- \sigma(V_s^y))\right] \der s\\
\end{align*}
By standard computations, since $d$ is dissipative, since $b$ is bounded and since $\sigma$ is Lipschitz and bounded,  for every $\varepsilon > 0$, we get
 \begin{align}\label{estimee WANG}
 \frac{|Y_{t \wedge \zeta_n}^{x,y}|^2}{T-s}  + (2L-1 - \varepsilon)\int_0^{t \wedge \zeta_n} \frac{|Y_s^{x,y}|^2}{|T-s|^2} \der s &\leq \frac{|x-y|^2}{T} + \frac{C}{\varepsilon} (t \wedge \zeta_n) \nonumber \\ 
 &~~~~ + 2\int_0^{t \wedge \zeta_n}  \left(\frac{Y_s^{x,y}}{T-s},\left[\sigma(\widetilde{V}_s)- \sigma(V_s^y)\right]\der \widetilde{W}_s \right).
 \end{align}
 By the Girsanov theorem, $(\widetilde{W}_s)_{s \leq t \wedge \zeta_n }$ is a standard Brownian motion under the probability measure $R_{t \wedge \zeta_n} \Pb$. So, taking expectation $\E^{t,n}$ with respect to $R_{t \wedge \zeta_n} \Pb$, we arrive at
 \begin{align}\label{estimee sous E t n}
 \E^{t,n} \int_0^{t \wedge \zeta_n} \frac{|Y_s^{x,y}|^2}{|T-s|^2} \der s \leq \frac{|x-y|^2}{T} + \frac{C}{\varepsilon}T, ~~~~ s \in [0,T), n \geq 1.
 \end{align}
By the definitions of $R_t$ and $\widetilde{W}_t$, we have for every $s \leq t \wedge \zeta_n$,
\begin{align*}
\log R_s &= -\int_0^s h(r) \der \widetilde{W}_r + \frac{1}{2}\int_0^s |h(r)|^2 \der r\\
&\leq  -\int_0^s h(r) \der \widetilde{W}_r  + C\left(1+\int_0^s\frac{L^2 |Y_r^{x,y}|^2}{|T-r|^2}\der r \right)
\end{align*} 
 And then, taking the expectation, we obtain with (\ref{estimee sous E t n}):
 \begin{align*}
 \E R_{t \wedge \zeta_n} \log R_{t \wedge \zeta_n} = \E^{t,n} \log R_{t \wedge \zeta_n} \leq \frac{|x-y|^2}{T} + \frac{C}{\varepsilon}T, ~~~~ t \in [0,T), n \geq 1.
 \end{align*}
By the martingale convergence theorem and the Fatou lemma, $(R_{s \wedge \zeta})_{s \in [0,T]}$ is a well-defined martingale with
\begin{align*}
 \E R_{t \wedge \zeta} \log R_{t \wedge \zeta} \leq \frac{|x-y|^2}{T} + \frac{C}{\varepsilon}T, ~~~~ t \in [0,T].
\end{align*}
To see that $(R_{s \wedge \zeta})_{s \in [0,T]}$ is a martingale, let $0 \leq s < t \leq T$. By the dominated convergence theorem and the martingale property of $(R_{t \wedge \zeta_n \wedge (T-1/n)})$, we have
\begin{align*}
\E(R_{t \wedge \zeta} | \mathscr{F}_s) &= \E(\lim_{n \rightarrow + \infty } R_{t \wedge \zeta_n \wedge (T-1/n)} | \mathscr{F}_s) = \lim_{n \rightarrow + \infty} \E(R_{t \wedge \zeta_n \wedge (T-1/n)} | \mathscr{F}_s) \\
& = \lim_{n \rightarrow + \infty} R_{s \wedge \zeta_n} = R_{s \wedge \zeta}.
\end{align*}
($2$)Since $\widetilde{W}_t$ is a standard Brownian motion up to $T \wedge \zeta$, it follows from (\ref{estimee WANG}) that 
\begin{align*}
\frac{n^2}{T} \Q(\zeta_n \leq t) \leq \E^{\Q}\frac{|Y_{t \wedge \zeta_n}^{x,y}|^2}{T-(t \wedge \zeta_n)} \leq \frac{|x-y|^2}{T} + \frac{C}{\varepsilon} T
\end{align*}
 holds for $n \geq 1$ and $t \in [0,T)$. By letting $n \nearrow + \infty$, we obtain $\Q(\zeta \leq t) = 0$ for all $t \in [0,T)$. This is equivalent to $\Q(\zeta  = T) = 1$ according to the definition of $\zeta$.\\
($3$) Let 
\begin{align*}
\overline{\zeta} = \inf \{ t \in [0,T]: Y_{t}^{x,y} = 0 \}
\end{align*}
and set $\inf \varnothing = + \infty$ by convention. We claim that $\overline{\zeta} \leq T$ and thus, $Y_T^{x,y} = 0$, $\Q$-a.s. Indeed, if for some $w \in \Omega$ such that $\overline{\zeta} > T$, by the continuity of the process we have
\begin{align*}
\inf_{t \in [0,T]} |Y_t^{x,y}|^2(w) > 0.
\end{align*}
So,
\begin{align*}
\int_0^T \frac{|Y_s^{x,y}|^2}{|T-s|^2} \der s = \infty
\end{align*}
holds on the set $\{ \overline{\zeta} > T \}$. Now, since inequality (\ref{estimee WANG}) still hold with $\zeta_n$ replaced by $\zeta$ and since $\zeta = T$ $\Q$-a.s., if we take the expectation with respect to $\Q$, then
\begin{align*}
\E^{\Q} \int_0^T \frac{|Y_s^{x,y}|^2}{|T-s|^2} \der s \leq \frac{|x-y|^2}{T} + \frac{C}{\varepsilon} T < +\infty.
\end{align*}
 Therefore $\Q(\overline{\zeta} > T) = 0$. Therefore, $Y_T^{x,y} = 0$, $\Q$-a.s.
\end{proof}

Hence the can apply the Girsanov theorem. Therefore, there exist a new probability measure $\Q$ under which $\widetilde{W}$ is a Brownian motion. Thus, under $\Q$, $\widetilde{V}$ has the law $\mu^x$ whereas under $\Pb$ it has the law $\widetilde{\mu}$. Of course $\mu^x$ and $\widetilde{\mu}$ are equivalent. We deduce that for every $\delta >0$,
\begin{align}\label{estimee rapport densite}
\int_{\mathscr{C}([0,T],\R^d)} \left( \frac{\der \mu^x}{\der \widetilde{\mu}} \right)^{2+\delta} \der \widetilde{\mu} &= \int_{\Omega} \left( \frac{\der \Q}{\der \Pb } \right)^{2+\delta} \der \Pb \nonumber \\
& =  \int_{\Omega} \left( \frac{\der \Q}{\der \Pb } \right)^{1+\delta} \der \Q \nonumber \\
&= \E^{\Q} \left(\exp\left((1+\delta)I_T-\frac{1+\delta}{2}\langle I,I\rangle_T\right)\right)  \nonumber \\
&\leq \sqrt{\E^{\Q} \left( \exp(2(1+\delta) I_T - 2(1+\delta)^2)\langle I,I\rangle_T ) \right)}\nonumber \\
& ~~~~~~\times \sqrt{\E^{\Q} \left( \exp( (2(1+\delta)^2 - (1+\delta)) \langle I,I \rangle_T ) \right)} \nonumber \\
&= \sqrt{\E^{\Q} \left( \exp( (1+\delta)(1+2\delta) \langle I,I \rangle_T ) \right)}.
\end{align}

We are going to show that we can pick $L > 0$ and $\delta > 0$ such that:
\begin{align*}
\E^{\Q} \left( \exp( (1+\delta)(1+2\delta) \langle I,I \rangle_T ) \right) < + \infty.
\end{align*}
Indeed, we have, for every $\iota >0$
\begin{align*}
\E^{\Q} (\exp ((1+\delta)(1+2\delta) \langle I,I \rangle_T)  ) \leq C \E^{\Q} \left( \exp \left( (1+\delta)(1+2\delta)|||\sigma^{-1} |||^2 \left((1+\iota)\int_0^T \frac{L^2 |Y_s^{x,y}|^2}{|T-s|^2} \right) \der s \right) \right)
\end{align*}
Now, by Itô's formula (using \ref{equation Y sous Q}):
\begin{align*}
\frac{L^2 |Y_t^{x,y}|^2}{T-t} &= \frac{L^2|x-y|^2}{T} + 2\int_0^t \left( \frac{L^2 Y_s^{x,y}}{T-s} , \left(d(V_s^y  + Y_s^{x,y}) - d(V_s^y) - \frac{L Y_s^{x,y}}{T-s}\right)\der s \right)\\
&  - 2\int_0^t \left( \frac{L^2 Y_s^{x,y}}{T-s} , (\sigma(V_s^y + Y_s^{x,y}) - \sigma(V_s^y))(b(V_s^y + Y_s^{x,y}) - b(V_s^y)) \der s \right)\\
&+ 2\int_0^t \left( \frac{L^2 Y_s^{x,y}}{T-s} ,(\sigma(V_s^y + Y_s^{x,y}) - \sigma(V_s^y))\der \widetilde{W}_s \right) + \int_0^t \frac{L^2 |Y_s^{x,y}|^2}{|T-s|^2} \der s \\
& + \int_0^t \frac{L^2}{T-s} \tr \left[ [\sigma(V_s^y  + Y_s^{x,y}) - \sigma(V_s^y)]\transpose[\sigma(V_s^y  + Y_s^{x,y}) - \sigma(V_s^y)]  \right] \der s.
\end{align*}
Therefore, since $d$ is dissipative, since $b$ is bounded and since $\sigma$ is Lipshitz and bounded,
\begin{align*}
\frac{L^2 |Y_t^{x,y}|^2}{T-t}  + 2\left(L-\frac{1}{2}\right)\int_0^t \frac{L^2|Y_s^{x,y}|^2}{|T-s|^2} \der s \leq \frac{L^2|x-y|^2}{T} + J_t + CL^2\int_0^t \frac{|Y_s^{x,y}|}{T-s} \der s
\end{align*}
where $J_t = 2\int_0^t \left( \frac{L^2 Y_s^{x,y}}{T-s} ,(\sigma(V_s^y + Y_s^{x,y}) - \sigma(V_s^y))\der \widetilde{W}_s \right) $. Note that 
\begin{align*}
\der \langle J,J \rangle_t &= 4 L^4\frac{|\transpose(Y_s^{x,y})(\sigma(V_s^y + Y_s^{x,y}) - \sigma(V_s^y))|^2}{|T-s|^2} \der t \\
& \leq 4 L^4 \Lambda^2 \frac{|Y_s^{x,y}|^2}{|T-s|^2}\der t.
\end{align*} 
Therefore, taking $t = T$, for every $\varepsilon >0$, for every $\gamma > 0$, we have
\begin{align*}
2\left(L -\frac{1}{2} - 
\frac{\varepsilon}{4} - \gamma L^2 \Lambda^2 \right) \int_0^T \frac{L^2|Y_s^{x,y}|^2}{|T-s|^2} \der s \leq \frac{L^2|x-y|^2}{T}  + \frac{C^2}{2\varepsilon} T + J_T - \frac{\gamma}{2} \langle J,J \rangle_T.
\end{align*}
Therefore, for $\varepsilon$ small enough, we obtain by multiplying by $\gamma$
\begin{align*}
2\gamma \left(L -\frac{1}{2} - 
\frac{\varepsilon}{4} - \gamma L^2 \Lambda^2 \right) \int_0^T \frac{L^2|Y_s^{x,y}|^2}{|T-s|^2} \der s \leq \frac{\gamma L^2|x-y|^2}{T}  + \frac{\gamma C^2}{2\varepsilon} T + \gamma J_T - \frac{\gamma^2}{2} \langle J,J \rangle_T.
\end{align*}
Therefore, 
\begin{align*}
\E^{\Q} \left(\exp \left( 2\gamma \left(L -\frac{1}{2} - 
\frac{\varepsilon}{4} - \gamma L^2 \Lambda^2 \right) \int_0^T \frac{L^2|Y_s^{x,y}|^2}{|T-s|^2} \der s \right) \right) \leq \exp\left(\frac{\gamma L^2|x-y|^2}{T}  + \frac{\gamma C^2}{2\varepsilon} T\right)
\end{align*}
Since by Hypothesis \ref{hypo f weakly dissipative, sigma invertible}, there exists $\lambda > 0$ such that
\begin{align*}
2(\lambda - \lambda^2 \Lambda^2) > |||\sigma^{-1}|||^2,
\end{align*}
it is enough to take $\iota > 0$, $\varepsilon > 0$, $\delta$ small enough, $\gamma L = \lambda$ and $\gamma$ small enough such that:
\begin{align*}
2\left(\gamma L -\frac{\gamma}{2} - \frac{\gamma \varepsilon}{4} 
- \gamma^2 L^2 \Lambda^2 \right) = 2\left(\lambda -\frac{\gamma}{2} 
- \frac{\gamma \varepsilon}{4} - \lambda^2 \Lambda^2 \right) >(1+\delta)(1+2\delta)(1+\iota) ||| \sigma^{-1}|||^2,
\end{align*}
which shows that:
\begin{align}\label{Novikov}
\E^{\Q} \left( \exp( (1+\delta)(1+2\delta) \langle I,I \rangle_T ) \right) < + \infty.
\end{align}
Therefore, with (\ref{estimee rapport densite}):
\begin{align}\label{estimee sur le rapport des mesures}
\int_{\mathscr{C}([0,T],\R^d)} \left( \frac{\der \widetilde{\mu}}{\der \mu^x} \right)^{2+\delta} \der \mu^x \leq \kappa_5.
\end{align}

We recall the following result (see for instance \cite{MATTINGLY_EXPONENTIAL}).

\begin{prop}\label{prop coupling 1}
Let $(\mu_1,\mu_2)$ be two probability measures on a same space $(E,\mathscr{E})$ then 
\begin{align*}
|| \mu_1 - \mu_2 ||_{TV} = \min \Pb (Z_1 \neq Z_2)
\end{align*}
where the minimum is taken on all coupling $(Z_1,Z_2)$ of $(\mu_1,\mu_2)$. Moreover, there exists a coupling which realizes the infimum. We say that it is a maximal coupling. It satisfies 
\begin{align*}
\Pb(Z_1 = Z_2,Z_1 \in \Gamma) = \mu_1 \wedge \mu_2 (\Gamma), ~~~\Gamma \in \mathscr{B}(E).
\end{align*}
Moreover, if $\mu_1$ and $\mu_2$ are equivalent and if 
\begin{align*}
\int_E \left( \frac{\der \mu_2}{\der \mu_1} \right)^{p+1} \der \mu_1 \leq C
\end{align*}
for some $p>1$ and $C>1$ then
\begin{align*}
\Pb(Z_1 = Z_2) = \mu_1 \wedge \mu_2 (E) \geq \left( 1 - \frac{1}{p} \right)\left( \frac{1}{pC}\right)^{1/(p-1)}.
\end{align*}
\end{prop}

Let us mention the following proposition which can be found in \cite{ODASSO_EXPONENTIAL_MIXING_NON_ADDITIVE} under the name of Corollary $1.5$.
\begin{prop}\label{prop coupling 2}
Let $E$ be a Polish space, $(\Omega, \mathscr{F}, \Pb)$ be a probability space and $(U_1,U_2,\widetilde{U})$ be three random variables on $(\Omega, \mathscr{F}, \Pb)$ taking value in $E$.

Then there exists a triple $(u_1,u_2,\widetilde{u})$ such that $(u_2,\widetilde{u})$ is a maximal coupling of $(\mathscr{D}(U_2),\mathscr{D}(\widetilde{U}))$ and such that the law of $(u_1,\widetilde{u})$ is $\mathscr{D}(U_1,\widetilde{U})$.
 \end{prop}
 
 Let us apply Proposition \ref{prop coupling 2} to the three random variables $(Y^{x,y}, V^x, \widetilde{V})$ : there exists $(\widehat{Y}^{x,y} , V^{1,x,y} , \widetilde{V}^{2,x,y} )$ such that $(V^{1,x,y} , \widetilde{V}^{2,x,y} )$ is a maximal coupling of $(\mu^x, \widetilde{\mu})$ on $[0,T]$ and $\mathscr{D}(\widehat{Y}^{x,y}, \widetilde{V}^{2,x,y})  = \mathscr{D}(Y^{x,y}, \widetilde{V})$. Therefore, 
 \begin{align*}
 \mathscr{D}(\widetilde{V}^{2,x,y} - \widehat{Y}^{x,y}) =  \mathscr{D}(\widetilde{V} - {Y}^{x,y}) = \mathscr{D}(V^y) := \mu^y,
 \end{align*}
 and
 \begin{align*}
 \mathscr{D}(\widehat{Y}^{x,y}) = \mathscr{D}(Y^{x,y}).
 \end{align*}
Note that the last inequality implies that $\widehat{Y}_T^{x,y} = 0, ~~~~\Pb\text{-a.s.}$

  Now remark that, $({V}^{1,y,x} , V^{2,x,y} := \widetilde{V}^{2,x,y}- \widehat{Y}^{x,y})$ is a coupling of $(\mu^x, \mu^y)$ and 
 \begin{align*}
 \Pb({V}_T^{1,y,x} = V_T^{2,x,y}) =  \Pb({V}_T^{1,y,x}  = \widetilde{V}_T^{2,x,y}) \geq \Pb({V}^{1,y,x}  = \widetilde{V}^{2,x,y}).
 \end{align*}
Now, remarking that $(\widehat{V}^x , \widehat{V})$ is a maximal coupling of $(\mu^x, \widetilde{\mu})$ and applying Proposition \ref{prop coupling 1} thanks to equation (\ref{estimee sur le rapport des mesures}) we get that
\begin{align*}
\Pb({V}^{1,y,x}  = \widetilde{V}^{2,x,y}) \geq \frac{1}{4 \kappa_5},
\end{align*}
which leads to 
\begin{align}\label{resultat step 2}
 \Pb({V}_T^{1,y,x} = V_T^{2,x,y}) \geq \frac{1}{4 \kappa_5}.
\end{align}

\textbf{Step 3.} We construct a coupling for any initial value and any date. For $x = y$, we set 
\begin{align*}
(U_t^{1} , U_t^{2} ) = (V_t^x, V_t^x), ~~~t \in [0,T].
\end{align*}
If $x$ or $y$ is not in $B_R$, then
\begin{align*}
(U_t^{1} , U_t^{2} ) = (V_t^x,\overline{V}_t^y), ~~~t \in [0,T]
\end{align*}
where $\overline{V}_t^y$ is the solution of equation (\ref{SDE avec f}) driven by a Wiener process $\overline{W}$ independent of $W$. The coupling of the laws of $V^x$, $V^y$ is defined as follows. Assuming that we have built $(U_t^{1} , U_t^{2})$ on $[0,nT]$, we take $(V^{1,x,y},V^{2,x,y})$ as above independent on $(U_t^{1} , U_t^{2})$ on $[0,nT]$ and set 
\begin{align*}
(U_{nT+t}^{1},U_{nT +t}^{2}) = \left(V_t^{1,U_{nT}^{1},U_{nT}^2},V_t^{2,U_{nT}^{1},U_{nT}^{2}} \right), ~~~~\forall t \in [0,T].
\end{align*}
The Markov property of $(U^1,U^2)$ implies that $(U^1,U^2)$ is a coupling of $(\mathscr{D}(V^x), \mathscr{D}(V^y))$ on $[0,(n+1)T]$.

We then define the following sequence of stopping times
\begin{align*}
L_m = \inf \{ l > L_{m-1}, U_{lT}^{1} , U_{lT}^2 \in B_R \}
\end{align*}
with $L_0 = 0$. Evidently, ($\ref{esperance exponentiel}$) can be generalized to two solutions and we have:
\begin{align*}
\E(e^{\alpha L_1 T}) \leq \kappa_3(1+|x|^2+|y|^2)
\end{align*}
and 
\begin{align*}
\E\left(e^{\alpha(L_{m+1} - L_m)T}) | \mathscr{F}_{L_m T} \right) \leq \kappa_3(1+|U_{L_m T}^1|^2 + |U_{L_m T}^2|^2).
\end{align*}
It follows that 
\begin{align*}
\E\left(e^{\alpha L_{m+1} T} \right) &\leq \kappa_3 \E \left(e^{\alpha L_m T}(1+ |U_{L_m T}^1|^2 + |U_{L_m T}^2|^2 )\right) \\
& \leq \kappa_3 (1+2 R^2)\E(e^{\alpha L_{m} T})
\end{align*} 
and
\begin{align*}
\E \left(e^{\alpha L_{m} T} \right) \leq \kappa_3^{m}(1+2R^2)^{m-1}(1+|x|^2 + |y|^2).
\end{align*}
Now set
\begin{align*}
\ell_0 = \inf \{ \ell, U^1_{(L_\ell +1)T} = U^2_{(L_\ell +1)T} \}.
\end{align*}
Since $U_{L_\ell T}^1$, $U_{L_\ell T}^2 \in B_R$, we have by ($\ref{resultat step 2}$), 
\begin{align*}
\Pb (\ell_0 > \ell + 1 | \ell_0 > \ell) \leq \left( 1- \frac{1}{\kappa_5} \right).
\end{align*}
Since $\Pb (\ell_0 > \ell +1) = \Pb(\ell_0 > \ell +1 | \ell_0 > \ell) \Pb(\ell_0 > \ell)$, we obtain
\begin{align*}
\Pb(\ell_0 > \ell) \leq \left( 1 - \frac{1}{4 \kappa_5} \right)^\ell.
\end{align*}
Then for $\gamma \geq 0$
\begin{align*}
\E(e^{\gamma L_{\ell_0} T}) &\leq \sum_{\ell \geq 0} \E(e^{\gamma L_{\ell} T} \mathds{1}_{\ell = \ell_0}) \\
& \leq \sum_{\ell \geq 0} \Pb(\ell = \ell_0)^{1-\gamma/\alpha} (\E(e^{\alpha L_\ell T}) )^{\gamma / \alpha}\\
& \leq \sum_{\ell \geq \ell_0} \left( 1 - \frac{1}{\kappa_5} \right)^{(\ell - 1)(1 - \gamma/\alpha)} [ \kappa_3^{\ell}(1+2R^2)^{\ell-1}(1+|x|^2 + |y|^2) ]^{\gamma/\alpha}.
\end{align*}
We choose $\gamma \leq \alpha$ such that
\begin{align*}
\left( 1 - \frac{1}{4 \kappa_5} \right)^{1 - \gamma / \alpha}[\kappa_3(1+2 R^2)]^{\gamma/\alpha} < 1
\end{align*}
and deduce that 
\begin{align*}
\E(e^{\gamma L_{\ell_0} T}) \leq \kappa_6(1+|x|^2 + |y|^2).
\end{align*}
Since
\begin{align*}
n_0 = \inf \left\{ k, U^1_{kT} = U_{kT}^2 \right\} \leq L_{\ell_0} +1
\end{align*}
it follows that 
\begin{align*}
\E(e^{\gamma n_0 T}) \leq \kappa_6(1+|x|^2+|y|^2)
\end{align*}
and
\begin{align*}
\Pb(U_{kT}^1 \neq U_{kT}^2) &= \Pb( k < n_0) \\
& \leq \kappa_6(1+|x|^2+|y|^2)e^{-\gamma k T}.
\end{align*}
Moreover, for  all $t\in [0,T)$
\begin{align*}
\Pb(U_{kT+t}^1 \neq U_{kT +t}^2) &\leq \Pb(U_{kT}^1 \neq U_{kT}^2) \\
&\leq \kappa_6(1 + |x|^2 + |y|^2)e^{-\gamma k T}\\
& \leq \kappa_7 (1+|x|^2 + |y|^2)e^{-\gamma (kT + t)}.
\end{align*}
We deduce for $\phi \in B_b(\R^d)$ 
\begin{align*}
|\mathscr{P}_t\left[\Phi\right](x) - \mathscr{P}_t\left[\Phi\right](y) | \leq C(1+|x|^2+|y|^2)e^{-\mu t}|\Phi|_0.
\end{align*}

\qed

\section{Proof of Lemma \ref{lemm convergence processus prenalise}}
We follow the proof of the part $3$ of \cite{MENALDI_STOCHASTIC_VARIATIONNAL_INEQUALITY_FOR_REFLECTED_DIFFUSION}. We need to adapt this proof because in our case, the set in which the process is reflected is not bounded. Therefore convergences are not uniform in $x$ anymore. In our case, the dissipativity of the process is enough to avoid the boundedness of $\G$. We will use the following notation $\beta(x) = (x-\Pi(x))$. Note that $F_n(x) = -2n \beta(x)$. We recall the following properties of the penalization term:
\begin{align}\label{prop de la distance convexe 0}
(x'-x , \beta(x)) \leq 0, ~~~\forall x \in \R^d, \forall x \in \overline{G},
\end{align}
\begin{equation}\label{prop de la distance convexe}
(x'-x , \beta(x) ) \leq (\beta(x') , \beta(x) ), ~~~\forall x,x' \in \R^n, 
\end{equation}
\begin{equation}\label{convexe admet un centre}
\exists c \in \G,~~\gamma >0, ~~\forall x \in \R^n,~~(x-c ,\beta(x)) \geq \gamma |\beta(x)|.
\end{equation}
In what follow, $C$ is a constant which may vary from line to line and which may depends on $x$ and $T$ but which is independent on $n$.

In a first time, we show that for any $1 \leq p < \infty$, there exists $C > 0$
\begin{align}\label{estimee n beta p}
\E\left[ \left( n \int_0^T |\beta(X_t^{x,n})| \der s\right)^p \right] \leq C, \forall n \in \N.
\end{align}
For $p=2$, Itô's formula gives us:
\begin{align*}
|X_t^{x,n}-c|^2 = |x-c|^2 + 2\int_0^t (X_s^{x,n}-c , f(X_s^{x,n})\der s + F_n(X_s^{x,n})\der s + \sigma(X_s^{x,n}) \der W_s)\\ + \int_0^t \sum_{i,j}(\sigma(X_s^{x,n}) \transpose \sigma(X_s^{x,n}))_{i,j} \der s.
\end{align*}
Using inequality (\ref{convexe admet un centre}), the fact that $\sigma$ is bounded and Remark (\ref{rema dissipativity implies previous conditions}) we deduce:
\begin{align}\label{estimee terme de penalisation avec le centre}
4n\gamma\int_0^t |\beta(X_s^{x,n})| \der s \leq |x-c|^2 + 2\int_0^t C \der s + 2\int_0^t (X_s^{x,n}-c , \sigma(X_s^{x,n}) \der W_s).
\end{align}
By a BDG inequality, using the fact that $|\sigma|$ is bounded:
\begin{align*}
\E \left[ \left| \int_0^T (X_s^{x,n}-c , \sigma(X_s^{x,n}) \der W_s ) \right|^p \right] &\leq C \E \left[ \left(\int_0^T |X_s^{x,n}-c|^2 \der s \right)^{p/2} \right].
\end{align*}
Now, by Lemma \ref{lemm estimates dissipativity},  it follows that the process $X^{x,n}$ has bounded moments of all orders independent of $n$, thus 
\begin{align*}
\E \left[ \left| \int_0^T (X_s^{x,n}-c , \sigma(X_s^{x,n}) \der W_s ) \right|^p \right] &\leq C,
\end{align*}
which leads to 
\[
\E\left[ \left( n \int_0^T |\beta(X_t^{x,n})|\der s \right)^p \right] \leq C, \forall n \in \N,
\]
for a constant $C$ which does not depend on $n$.

Now we prove that $\forall p > 2$, 
\[ \E \left( \sup_{0 \leq t \leq T } |\beta(X_t^{x,n})|^p \right) \leq \frac{C}{n^{p/2-1}}. \]

We apply Itô's formula to the function $\varphi(x) = |x-\Pi(x)|^p$ where $\beta(x) = x-\Pi(x)$. Note that $F_n(x) = -2n \beta(x)$. It is well known that for a regular boundary,  for all $p > 2$, $\varphi$ is $\mathscr{C}^2$ on $\R^d$ and $\nabla \varphi(x) = 2(x-\Pi(x))$. We recall the following formulas for the derivatives of $\varphi$, 
\[ \nabla \varphi(x) = p|\beta(x)|^{p-2}\beta(x),\]
\[ \nabla^2 \varphi(x) = p|\beta(x)|^{p-2} \nabla \beta(x) + p(p-2)|\beta(x)|^{p-4}(\beta(x)	\transpose \beta(x)).\]
As $\nabla \beta(x)$ is a numerical matrix one can deduce the following inequality: 
\[|\nabla^2 \varphi(x)|  \leq C|\beta(x)|^{p-2},\]
for a constant $C$ which depends only on $p$ and $d$.

We use Itô's formula, for all $p >2$,
\begin{align*}
\varphi(X_t^{x,n}) = &\int_0^t (\nabla \varphi(X_s^{x,n}) , d(X_s^{x,n})+b(X_s^{x,n})+F_n(X_s^{x,n})) \der s\\ 
 &+ \int_0^t (\nabla \varphi(X_s^{x,n}) , \sigma(X_s^{x,n}) ) \der W_s \\
 &+\frac{1}{2}\int_0^t \sum_{i,j} (\nabla^2 \varphi(X_s^{x,n}))_{i,j} (\sigma(X_s^{x,n}) \transpose \sigma(X_s^{x,n}))_{i,j}  \der s.
\end{align*}
Therefore,
\begin{align}\label{eq estimee sur la distance varphi}
\varphi(X_t^{x,n}) + 2pn\int_0^t \varphi(X_s^{x,n}) \der s &\leq  \int_0^t (\nabla\varphi(X_s^{x,n}),d(X_s^{x,n})+b(X_s^{x,n})) \der s  \\
& ~~~  + \int_0^t (\nabla \varphi(X_s^{x,n}) ,\sigma(X_s^{x,n}) \der W_s )+ C\int_0^t |\beta(X_s^{x,n})|^{p-2} \der s. \nonumber
\end{align}
From Young's inequality: $ab \leq a^q/q + b^{q'}/q'$ for some real numbers $q$ and $q'$ such that $1/q + 1/q' = 1$, we choose  $q=p/(p-2)$ and $q' = p/2$ so that, for $\alpha > 0 $:
\begin{align*}
|\beta(X_s^{x,n})|^{p-2} &= \alpha n^{(p-2)/p}|\beta(X_s^{x,n})|^{p-2} \times \frac{1}{\alpha n^{(p-2)/p}}\\
&\leq \alpha^{p/(p-2)} \frac{p-2}{p}  n|\beta(X_s^{x,n})|^{p}  + \frac{2}{p} \left(\frac{1}{\alpha n^{(p-2)/p}}\right)^{p/2}\\
& \leq  \alpha^{p/(p-2)} \frac{p-2}{p}  n\varphi(X_s^{x,n}) + \frac{2}{p} \frac{1}{\alpha^{p/2} n^{(p-2)/2}},
\end{align*}
and another Young's inequality applied with this time $q=p/(p-1)$ and $q' = p$ gives us: 
\begin{align*}
|(\nabla \varphi (X_s^{x,n}) , d(X_s^{x,n}) + b(X_s^{x,n}))| &\leq p|\beta (X_s^{x,n})|^{p-1}\times|(d(X_s^{x,n}) + b(X_s^{x,n})| \\
& \leq \alpha^{p/(p-1)}n(p-1)p^{1/p}|\beta(X_s^{x,n})|^{p} \\
&~~~~~~~~~~+ |d(X_s^{x,n})+b(X_s^{x,n})|^p/(pn^{p-1}\alpha^p)\\
& \leq \alpha^{p/(p-1)}n(p-1)p^{1/p} \varphi(X_s^{x,n})\\
& ~~~~~~~~~~+ (1+|X_s^{x,n}|^{p \nu})/(pn^{p-1}\alpha^p).
\end{align*}
Therefore using the second inequality of Lemma \ref{lemm estimates dissipativity} and the two above inequality we deduce, for $\alpha$ small enough:
\[
\E \left(n\int_0^t\varphi(X_s^{x,n}) \der s \right)  \leq  C\left(\frac{1}{n^{p-1}} +\frac{1}{n^{(p-2)/2}}\right)t,
\]
therefore,
\begin{align}\label{eq 14}
\E \left(\int_0^T\varphi(X_s^{x,n}) \der s \right)  \leq  \frac{C}{n^{p/2}}.
\end{align}

Now we come back to equation (\ref{eq estimee sur la distance varphi}).
 Taking the supremum over time and the expectation and using a BDG inequality we get:
 \begin{align}\label{eq C5}
 \E \sup_{0 \leq t \leq T} \varphi (X_t^{x,n}) &\leq \E \int_0^T |\nabla \varphi(X_s^{x,n})| \times |d(X_s^{x,n}) + b(X_s ^{x,n})| \der s \\& ~~~~~+ C\E \left[ \left(\int_0^T |\nabla \varphi (X_s^{x,n}) \sigma(X_s^{x,n}) |^2\der s \right)^{1/2} \right] +C \E\int_0^T |\beta(X_s^{x,n})|^{p-2} \der s .\nonumber
 \end{align}
 We call respectively $I_1$, $I_2$ and $I_3$ the three terms of the right hand side of (\ref{eq C5}). We have 
 \begin{align*}
 I_1 & \leq C\sqrt{\E \int_0^T |\nabla \varphi(X_s^{x,n})|^2 \der s} \times \sqrt{\E \int_0^T \left|1+|X_s^{x,n}|^{\nu} \right|^2 \der s}  \\
 & \leq C \sqrt{\E \int_0^T |\beta(X_s^{x,n})|^{2p-2} \der s} \\
 & \leq \frac{C}{n^{(p-1)/2}},
 \end{align*}
by using inequality (\ref{eq 14}) and Lemma \ref{lemm estimates dissipativity}.
 We also have 
  \begin{align*}
 I_2 &\leq  C\E \left[ \left( \int_0^T |\beta(X_s^{x,n}|^{2p-2} \der s \right)^{1/2} \right]\\
 &\leq  \frac{1}{\sqrt{2}} \E\sup_{0 \leq t \leq T} |\beta(X_s^{x,n})|^p + C'\E \left[ \int_0^T |\beta (X_s^{x,n}) |^{p-2} \der s \right],
 \end{align*}
 thanks to Young's inequality. Applying inequality (\ref{eq 14}) to the second member gives us:
 \begin{align*}
 I_2 \leq \frac{1}{\sqrt{2}} \E\sup_{0 \leq t \leq T} |\beta(X_s^{x,n})|^p + \frac{C}{n^{(p-2)/2}}.
 \end{align*}
Finally applying inequality (\ref{eq 14}) once again gives us:
\begin{align*}
I_3 \leq \frac{C}{n^{(p-2)/2}}.
\end{align*} 
The above estimates of $I_1$, $I_2$ and $I_3$ give us the following inequality, for all $p > 2$
\begin{align}\label{estimee E sup phi X penalise}
\E \sup_{0 \leq t \leq T} \varphi (X_t^{x,n}) &\leq \frac{C}{n^{(p-2)/2}}.
\end{align}
 Now, we claim that for all $1 \leq p < +\infty$, $0 < 2q < p$, $n,m \in \N$,
 \begin{align}\label{estimee E beta m beta n}
 \E\left[ \left( m\int_0^T|\transpose \beta(X_s^{x,n})  \beta(X_s^{x,m})| \der s\right)^p \right] \leq \frac{C}{n^{q}}.
 \end{align}
Indeed, we have
\begin{align*}
m\int_0^T |\transpose \beta(X_s^{x,n})  \beta(X_s^{x,m})| \der s &\leq AB\\
& := (\sup_{0 \leq s \leq T} |\beta(X_s^{x,n})|)\left( m \int_0^T |\beta(X_s^{x,m})| \der s \right).
\end{align*} 
 Since, for $r >2$
 \begin{align*}
 E\left[ (AB)^p  \right] \leq (E(A^{rp}))^{1/r}(E(B^{r'p}))^{1/r'}, ~~~~~r' = \frac{r}{r-1},
 \end{align*}
 from (\ref{estimee n beta p}) and (\ref{estimee E sup phi X penalise}), we get
 \begin{align*}
 \E((AB)^p)\leq   \frac{C}{n^{(rp-2)/(2r)}} = \frac{C}{n^{p/2 - 1/r}},
 \end{align*}
 which implies (\ref{estimee E beta m beta n}) for $r$ large enough.
 
 Now we will prove that if $2 < 2q < p < \infty$, there exists a constant $C$ independent on $n$ and $m$ such that 
 \begin{align}\label{Xn moins Xm de Cauchy pour la norme E sup puissance p}
 \E \left[ \sup_{0 \leq s \leq T} |X_s^{x,n}-X_s^{x,m}|^p \right] \leq C \left(\frac{1}{n}+\frac{1}{m}\right)^q, ~~~~\forall n, m \in \N^*.
 \end{align}
Indeed, applying Itô's formula, for all $0 \leq t \leq T < +\infty$:
\begin{align*}
|X_t^{x,n}-X_t^{x,m}|^2 &= 2\int_0^t \transpose(X_s^{x,n} - X_s^{x,m})((d+b)(X_s^{x,n})-(d+b)(X_s^{x,m}))\der s \nonumber \\
&~~~~~ - 4n\int_0^t \transpose(X_s^{x,n}-X_s^{x,m})\beta(X_s^{x,n})\der s \nonumber \\
&~~~~~ + 4m\int_0^t \transpose(X_s^{x,n}-X_s^{x,m})\beta(X_s^{x,m})  \der s\\
&~~~~~ +2\int_0^t \transpose(X_s^{x,n}-X_s^{x,m})(\sigma(X_s^{x,n}) - \sigma(X_s^{x,m})) \der W_s \nonumber \\
&~~~~~ +\int_0^t \sum_{i} [(\sigma(X_s^{x,n}) - \sigma(X_s^{x,m}))\transpose(\sigma(X_s^{x,n}) - \sigma(X_s^{x,m}))]_{i,i}\der s. \nonumber
\end{align*}
By hypothesis on $d$, $b$ and $\sigma$ and thanks to equation (\ref{prop de la distance convexe}) we obtain
\begin{align}\label{eq Xn - Xm}
|X_t^{x,n}-X_t^{x,m}|^2 &\leq  C\int_0^t |X_s^{x,n}-X_s^{x,m}|^2\der s \nonumber\\
&~~~~~~~ + 4\E n\int_0^t \transpose\beta(X_s^{x,n})\beta(X_s^{x,m})\der s + 4\E m\int_0^t \transpose\beta(X_s^{x,n})\beta(X_s^{x,m})\der s  \nonumber\\ &~~~~~~~+2\int_0^t \transpose(X_s^{x,n}-X_s^{x,m})(\sigma(X_s^{x,n}) - \sigma(X_s^{x,m})) \der W_s.
\end{align}

Now, if $I_t$ denotes the stochastic integral in $(\ref{eq Xn - Xm})$ , we have for $r > 1$
\begin{align*}
\E \left[ \sup_{0 \leq s \leq t} |I_s|^r \right]&\leq CE \left[ \left( \int_0^t | X_s^{x,n} - X_s^{x,m}  |^2 \der s \right)^{r/2}\right]. 
\end{align*}
Therefore, by virtue of property (\ref{estimee E beta m beta n}) and for a new constant $C$, $2r = p$, $q < p/2$ and $0\leq t \leq T$,  we deduce from (\ref{eq Xn - Xm})
\begin{align*}
\E \sup_{0 \leq s \leq t}\left[  |X_s^{x,n} - X_s^{x,m}|^p \right] &\leq C\int_0^t E \left[|X_s^{x,n}-X_s^{x,m}|^p\right] \der s +C\left(\frac{1}{n^q} + \frac{1}{m^q}\right) \\
&\leq C\int_0^t E \sup_{0 \leq u \leq s} \left[|X_u^{x,n}-X_u^{x,m}|^p\right] \der s +C\left(\frac{1}{n^q} + \frac{1}{m^q}\right),
\end{align*}
which implies (\ref{Xn moins Xm de Cauchy pour la norme E sup puissance p}) after using Gronwall's Lemma.
inequality (\ref{Xn moins Xm de Cauchy pour la norme E sup puissance p}) shows us that the process $(X^{x,n})_n$ is Cauchy in $\mathscr{S}^p$, $p>2$. Of course, taking the power $1/r$ for $r >1$ in equation (\ref{Xn moins Xm de Cauchy pour la norme E sup puissance p}) shows us that $(X^{x,n})_n$ is Cauchy in $\mathscr{S}^p$, for all $p \geq 1$.  Therefore we can define in  $\mathscr{S}^p$, for all $p \geq 1$:
\begin{align*}
X^x := \lim_{n \rightarrow + \infty} X^{x,n}.
\end{align*}
Furthermore, defining
\begin{align*}
\eta_t^{x,n} = -2n\int_0^t \beta(X_s^{x,n}) \der s,
\end{align*}
and remarking that the process $\eta^{x,n}$ satisfies the following equation
\begin{align*}
\eta_t^{x,n}  = X_t^{x,n}  - \int_0^t (d+b)(X_s^{x,n}) \der s - \int_0^t \sigma(X_s^{x,n})\der W_s,
\end{align*}
we can define in   $\mathscr{S}^p$, for all $p \geq 1$:
\begin{align*}
\eta^{x} :=  \lim_{n \rightarrow + \infty} \eta^{x,n}.
\end{align*}
Clearly, estimate (\ref{estimee n beta p}) shows that $\eta^x$ has locally bounded variation almost surely, and condition (\ref{estimee E sup phi X penalise}) implies that $X^x$ belongs to $\overline{G}$ almost surely. Now remark that property (\ref{prop de la distance convexe 0}) implies that, for all $T > 0$, for all progressively measurable process $z$ taking values in the closed set $\overline{G}$
\begin{align}\label{SVI inequality}
\int_0^T \transpose(X_s^{x,n}-z_s)\der \eta_s^{x,n} =  -2n\int_0^T\transpose(X_s^{x,n}-z_s) \beta(X_s^{x,n})   \der s \leq 0.
\end{align}
Now we would like to pass to the limit into the previous inequality. For that purpose let us recall the following deterministic Lemma.
\begin{lemm}
Let $y_n$ be a sequence of functions in $\mathscr{C}([0,T],\R^k)$ which converges uniformly to $y$. Let $\eta_n$ be a sequence of functions in $\mathscr{C}([0,T],\R^k)$ which converges uniformly to $\eta$ and such that there exists $C>0$ such that $||\eta_n||_{TV} \leq C$. Then
\begin{align*}
||\eta||_{TV} \leq C \text{~~~and~~~~}\int_0^T y_n \der \eta_n \underset{n\rightarrow + \infty}{\longrightarrow} \int_0^T y \der \eta.
\end{align*} 
\end{lemm}
\begin{proof}
See Lemma 5.7 in \cite{GEGOUT-PETIT_PARDOUX_EDSR_REFLECHIES_DANS_UN_CONVEXE}.
\end{proof}
Now we claim that we can extract a subsequence of $\left\{\int_0^T \transpose(X_s^{x,n}-z_s)\der \eta_s^{x,n}\right\}_n$ which converges $\Pb$-a.s. to $\int_0^T \transpose(X_s^{x}-z_s)\der \eta_s^{x}$. First, by equation (\ref{estimee terme de penalisation avec le centre}), we have 
\begin{align*}
2\gamma ||\eta^{x,n}||_{TV} \der s \leq |x-c|^2 + 2\int_0^t C \der s + 2\int_0^t (X_s^{x,n}-c , \sigma(X_s^{x,n}) \der W_s).
\end{align*}
Now remark that the right member of the inequality converges in probability to $|x-c|^2 + 2\int_0^t C \der s + 2\int_0^t (X_s^{x}-c , \sigma(X_s^{x}) \der W_s)$. Therefore, there exists a subsequence $(\phi(n))_n$ such that the right member of the inequality converges $\Pb$-a.s. and then $||\eta^{x,\phi(n)}||_{TV}$ is bounded $\Pb$-a.s. Furthermore, as $\E\sup_{0 \leq t \leq T}|X_t^{x,n}-X_t^{x}|^2 \underset{n \rightarrow + \infty}{ \longrightarrow} 0$, and $\E\sup_{0 \leq t \leq T}|\eta_t^{x,n}-\eta_t^{x}|^2 \underset{n \rightarrow + \infty}{ \longrightarrow} 0$  we can extract a subsequence $(\psi(n))_n$ of $(\phi(n))_n$ such that $X^{x,\psi(n)}$ converge uniformly to $X^{x}$ $\Pb$-a.s. and $\eta^{x,\psi(n)}$ converge uniformly to $\eta^{x}$ $\Pb$-a.s. Now, it is enough to apply the previous deterministic Lemma $\Pb$-a.s. to obtain the result, namely, for all $T > 0$, for all progressively measurable process $z$ taking values in the closed set $\overline{G}$
\begin{align*}
\int_0^T \transpose(X_s^{x}-z_s)\der \eta_s^{x}  \leq 0.
\end{align*}

By Lemma $2.1$ in \cite{GEGOUT-PETIT_PARDOUX_EDSR_REFLECHIES_DANS_UN_CONVEXE}, this implies that 
\begin{align*}
\int_0^T \mathds{1}_{\left\{ X_s^{x} \in G \right\} } \der \eta^{x}_s = 0, ~~~~~\Pb\text{-a.s.}
\end{align*}
and 
\begin{align*}
\eta_t^{x} = \int_0^t \nabla \phi(X_s^{x}) \der ||\eta^{x}||_{TV,s}.
\end{align*}
We define $\der K_s := \der ||\eta^x||_{TV,s}$. This shows that $(X^x,K^x)$ satisfies the SDE (\ref{SDE reflected}). 

Now we prove that the solution of the SDE (\ref{SDE reflected}) is unique. Let us assume that $(X^1,K^1)$ and $(X^2,K^2)$ are two solutions of (\ref{SDE reflected}) such that $K^1$ and $K^2$ have bounded variation on $[0,T]$, for all $T>0$ and such that for all $i =1,2$, for all continuous and progressively measurable process $z$ taking values in the closure $\overline{G}$ we have,
\begin{align*}
\int_0^T \transpose(X^i_s-z_s) \der \eta^{i}_s \leq 0,
\end{align*}  
where $\eta^i_t = \int_0^t \nabla\phi(X_s^i) \der K_s^i$.
Applying Itô's formula, one has
\begin{align*}
|X^1_t - X_t^2|^2 = &2\int_0^t \transpose (X_s^1-X_s^2)(d(X_s^1) - d(X_s^2)) \der s + 2\int_0^t \transpose (X_s^1-X_s^2)(b(X_s^1) - b(X_s^2)) \der s \\
&+2\int_0^t \transpose (X_s^1-X_s^2)(\der \eta_s^1 - \der \eta_s^2) \der s + 2\int_0^t \transpose (X_s^1-X_s^2)(\sigma(X_s^1)-\sigma(X_s^2))\der W_s\\
&+\int_0^t \tr [(\sigma(X_s^1) -\sigma(X_s^2)) \transpose(\sigma(X_s^1) -\sigma(X_s^2)) ] \der s.
\end{align*}
Since, $2\int_0^t \transpose (X_s^1-X_s^2)(\der \eta_s^1 - \der \eta_s^2) \der s \leq 0$, we easily get
\begin{align*}
\E\left[\sup_{0 \leq t \leq T}|X_t^1-X_t^2|^2\right] &\leq C\E \left[ \int_0^T  |X_s^1-X_s^2|^2 \der s \right] \\
&\leq C \left[ \int_0^T  \E \sup_{0 \leq u \leq s}|X_u^1-X_u^2|^2 \der s \right].
\end{align*}
Applying Gronwall's lemma allows us to conclude. 
\qed

\section*{Acknowledgements}
The author would like to thank his Ph.D. adviser Ying Hu for his helpful remarks and comments during the preparation of this paper. The author also thanks the Lebesgue Center of Mathematics for its support (program "Investissements d'avenir" --- ANR-11-LABX-0020-01) and the anonymous referee for the helpful comments.

\newpage

\bibliographystyle{hplain}
\bibliography{biblio}
\end{document}